\documentclass[10pt, oneside]{amsart}

\usepackage[utf8]{inputenc}
\usepackage{amsmath}
\usepackage{amssymb}
\usepackage{amsfonts}
\usepackage{amsthm}
\usepackage{thmtools}
\usepackage{enumerate}
\usepackage[top=2cm, bottom=2cm, left=2cm, right=2cm]{geometry}
\usepackage{graphicx}
\usepackage[all,cmtip]{xy}
\usepackage{titlesec}
\usepackage{hyperref}
\usepackage{rotating}
\usepackage{xcolor}
\usepackage{setspace}
\usepackage{enumitem} 
\usepackage{verbatim}
\usepackage{cancel}
\usepackage{tikz-cd}
\usepackage{makecell}

\DeclareSymbolFontAlphabet{\amsmathbb}{AMSb}

\usepackage[math-style=TeX, bold-style=TeX]{unicode-math}

\hypersetup{colorlinks=true, linkcolor=blue, citecolor = blue, filecolor = blue, urlcolor = blue}

\declaretheoremstyle[bodyfont=\normalfont]{normalbody}

\declaretheorem[numberwithin=subsection,name=Theorem]{theorem}
\declaretheorem[sibling=theorem,style=normalbody,name=Definition]{definition}
\declaretheorem[sibling=theorem,name=Corollary]{corollary}
\declaretheorem[sibling=theorem,name=Lemma]{lemma}
\declaretheorem[sibling=theorem,name=Proposition]{proposition}

\declaretheorem[sibling=theorem,style=normalbody,name=Remark]{remark}
\declaretheorem[sibling=theorem,style=normalbody,name=Remarks]{remarks}

\declaretheorem[sibling=theorem,style=normalbody,name=Fact]{fact}

\declaretheorem[numbered=no,name=Theorem]{theorem-no}

\numberwithin{equation}{section}

\newcommand{\Z}{\mathbb{Z}}
\newcommand{\N}{\mathbb{N}}
\newcommand{\Q}{\mathbb{Q}}
\newcommand{\R}{\mathbb{R}}

\newcommand{\G}{\mathbb{G}}

\renewcommand{\P}{\mathbb{P}}

\newcommand{\pint}[2]{\left\langle#1,#2\right\rangle}

\newcommand{\s}[1]{\mathcal{#1}}

\newcommand{\Hom}{\operatorname{Hom}}
\newcommand{\Ext}{\operatorname{Ext}}

\newcommand{\Gal}{\operatorname{Gal}}
\newcommand{\Br}{\operatorname{Br}}

\newcommand{\coker}{\operatorname{coker}}

\newcommand{\cone}{\operatorname{cone}}

\newcommand{\Cores}{\operatorname{Cores}}
\newcommand{\inv}{\operatorname{inv}}

\newcommand{\Pic}{\operatorname{Pic}}
\newcommand{\Alb}{\operatorname{Alb}}
\newcommand{\NS}{\operatorname{NS}}
\newcommand{\et}{\operatorname{\'et}}
\newcommand{\sm}{\mathrm{sm}}
\newcommand{\fppf}{\mathrm{fppf}}

\newcommand{\CH}{\operatorname{CH}}
\newcommand{\sHom}{\mathcal{H\mkern-6mu o\mkern-1mu m\mkern.5mu}}
\newcommand{\sExt}{\mathcal{E\mkern-4mu x\mkern-0.5mu t\mkern.5mu}}

\newcommand{\Spec}{\operatorname{Spec}}

\newcommand{\id}{\operatorname{id}}

\newcommand{\im}{\mathrm{im}\mkern1mu}
\newcommand{\D}{\mathbf{D}}

\newcommand{\tors}{\mathrm{tors}}

\newcommand{\cts}{\mathrm{cts}}
\newcommand{\LCA}{\mathbf{LCA}}
\newcommand{\Haus}{\mathrm{Haus}}
\newcommand{\pseudo}{\mathrm{I}_\tau}
\newcommand{\vsim}{\rotatebox[origin=c]{-90}{$\sim$}}

\DeclareFontFamily{U}{wncy}{}
\DeclareFontShape{U}{wncy}{m}{n}{<->wncyr10}{}
\DeclareSymbolFont{mcy}{U}{wncy}{m}{n}
\DeclareMathSymbol{\Sha}{\mathord}{mcy}{"58}

\DeclareFontFamily{U}{wncy}{}
\DeclareFontShape{U}{wncy}{m}{n}{<->wncyr10}{}
\DeclareSymbolFont{mcy}{U}{wncy}{m}{n}
\DeclareMathSymbol{\Cha}{\mathord}{mcy}{"51}

\titleformat{\section}[block]
{\normalsize\sc}{\thesection}{1em}{\centering\normalsize}
\titleformat{\subsection}[block]
{\normalsize\bf}{\thesubsection}{1em}{\normalsize}
\titleformat{\paragraph}[display]
{\normalsize\bf}{\theparagraph}{1em}{\normalsize}

\tikzcdset{
	arrow style=math font,
}

\tikzset{
	symbol/.style={
		draw=none,
		every to/.append style={
			edge node={node [sloped, allow upside down, auto=false]{$#1$}}}
	}
}

\expandafter\def\expandafter\normalsize\expandafter{%
	\normalsize%
	\setlength\abovedisplayskip{.5em}%
	\setlength\belowdisplayskip{.5em}%
	\setlength\abovedisplayshortskip{.5em}%
	\setlength\belowdisplayshortskip{.5em}%
}

\title{Lichtenbaum-van Hamel duality for singular varieties over $p$-adic fields}

\author{Felipe Rivera-Mesas}
\address{Departamento de Matemáticas, Facultad de Ciencias, Universidad de Chile}
\email{felipe.rivera.m@ug.uchile.cl}
\thanks{ 
	\mbox{\hspace{1.3em}} {\sl MSC codes (2020):} 14G20 (Algebraic Geometry - Local ground fields), 14F22 (Brauer group of schemes), 11S25 (Galois Cohomology); \\
	\mbox{\hspace{1.3em}} This work was partially supported by the Beca de Doctorado Nacional ANID folio 21210171, awarded in 2021.}

\begin{document}
	
	\begin{abstract}
		In this article, we extend the van Hamel-Lichtenbaum duality theorem to (not necessarily smooth) proper and geometrically integral varieties defined over a $p$-adic field $k$. More precisely, we prove that for such variety $X$ there exists a natural continuous perfect pairing
			\[ \Br_1(X)\times H_0(X,\Z)_\tau^{\wedge} \to \Q/\Z, \]
		where $\Br_1(X):=\ker(\Br(X)\to\Br(\overline{X}))$ is the algebraic Brauer group of $X$, $H_0(X,\Z)_\tau$ is the zeroth group of truncated homology $\Hom_{D(k_{\sm})}(\tau_{\leq 1}R\phi_*\G_{m,X},\G_{m,k})$, $\phi$ is the structure morphism of $X$, and $(-)^{\wedge}$ is the profinite completion functor.
	\end{abstract}
	
	\maketitle
	
\section{Introduction}

In 1969, Lichtenbaum found an explicit description for the Pontryagin dual of the Brauer-Grothendieck group $\Br(X):=H^2_{\et}(X,\G_m)$ of a projective, geometrically connected and smooth curve $X$ over a $p$-adic field $k$. More precisely, he defined a nondegenerate pairing
\begin{equation} \label{lichtenbaum pairing}
	\Br(X) \times \Pic(X) \to \Br(k) \overset{\inv}{\cong} \Q/\Z,
\end{equation}
where $\inv:\Br(k)\overset{\sim}{\to}\Q/\Z$ is the invariant map from local class field theory. Thus, he showed that the Pontryagin dual $\Br(X)^*:=\Hom(\Br(X),\Q/\Z)$ is isomorphic to $\Pic(X)^\wedge$ (see \cite{L69}). In fact, he actually constructed a general pairing
\[ \Br(X) \times \CH_0(X) \to \Br(k) \] 
for a proper and geometrically integral curve $X$ defined over an arbitrary field $k$. When $X$ is also smooth and $k$ is a non-archimedean local field, this pairing becomes the pairing \eqref{lichtenbaum pairing} because, in this case, $\CH_0(X)\cong\Pic(X)$ and $\Br(k)\cong\Q/\Z$ via the invariant map. Indeed, when $X$ is a locally factorial (e.g. regular, see \cite{AB}) scheme, we have an isomorphism $\CH_0(X)\cong\Pic(X)$ (cf. \cite[Corollaire 21.6.10]{EGAIV}).   \\

Now, we briefly recall the main ideas of Lichtenbaum's proof of the nondegeneracy of pairing \eqref{lichtenbaum pairing}. Let us recall that when $k$ is a non-archimedean local field and $X$ is a proper and geometrically integral $k$-variety, we have the following exact sequence
\[ 0 \to \frac{\Pic(X^s)^\Gamma}{\Pic(X)} \to \Br(k) \to \Br_1(X) \to H^1(k,\Pic_{X/k}) \to 0 \]
where $\Gamma=\Gal(k^s/k)$ is the Galois group of some separable closure $k^s$ of $k$ (cf. \cite[Proposition 5.4.2]{BG}). Moreover when $X$ is a curve, we have that $\Br_1(X)=\Br(X)$ (cf. \cite[Theorem 5.6.1.iv)]{BG}). On the other hand, we have the following exact sequence
\[ 0 \to \Pic^0_{X/k} \to \Pic_{X/k} \to \NS_{X/k} \to 0. \]
These two exact sequences are the theoretical framework in which Lichtenbaum developed his duality results. In this context, the main issue is to determine the Pontryagin dual of $H^1(k,\Pic_{X/k})$. In Lichtenbaum's setting, it is crucial that the Albanese variety $\Alb_{X/k}^0$ is an abelian variety because he used Tate's local duality theorem for abelian varieties over local fields for constructing the pairing above. More precisely, he proved that $H^1(k,\Pic_{X/k})$ and $H^0(k,\Alb^0_{X/k})\cong\CH_0(X)^0$, the group of classes of 0-cycles of degree 0 modulo rational equivalence, are dual (the author uses the notation $\Pic_0(X)$ to denote $\CH_0(X)^0$). Thus, using the exact sequences above, he could prove that the pairing \eqref{lichtenbaum pairing} is nondegenerate and, thereby, $\Br(X)$ and $\Pic(X)^\wedge$ are dual. Note that $\Pic(X)$ appears in relation to the Albanese variety $\Alb_{X/k}^0$ and not to the Picard scheme $\Pic_{X/k}$. \\

In 2004, van Hamel generalized Lichtenbaum's duality theorem for smooth and proper varieties $X$ over a $p$-adic field $k$ of arbitrary dimension (see \cite[Theorem 1]{vH04}). Note that van Hamel's result assumes smoothness on $X$. In his work, van Hamel redefined the pairing \eqref{lichtenbaum pairing} as a particular case of a Yoneda pairing. Moreover, van Hamel introduced a new homology for varieties with which this Yoneda pairing becomes a pairing between this new homology and sheaf cohomology. When we stand in the case of curves (Lichtenbaum's situation), the group $\CH_0(X)$ is realized from a homological group in degree 0 in this new homology. In general, this homological group is isomorphic to the $k$-points of the Albanese scheme $\Alb_{X/k}$ (cf. \cite[Theorem 3.6]{vH04}) defined by Ramachandran (see \cite{Ra01}). Thus, van Hamel proved that for a smooth and proper variety $X$ over a $p$-adic field the Pontryagin dual of $\Br(X)$ is isomorphic to the profinite completion of $\Alb_{X/k}(k)$. Finally, it is important to mention that van Hamel worked with technical machinery (such as derived categories) for proving the perfection of the Yoneda pairing. In this way, he gives a more conceptual proof of Lichtenbaum's duality theorem. \\

In this article, we extend van Hamel's generalization of Lichtenbaum's duality to the general case of any (possibly singular) proper variety over a $p$-adic field. In order to do that, for a proper variety $X$ over a $p$-adic field $k$ with structure map $\phi:X\to\Spec k$, we define the groups of truncated homology $H_i(X,\Z)_{\tau}$ of $X$ as the group of derived homomorphisms $R\Hom_{k_{\sm}}(\tau_{\leq 1}R\phi_*\G_m,\G_m[i])$. The main result is the following.

\begin{theorem-no}[Corollary \ref{lichtenbaum duality}]
	We have a perfect continuous pairing
	\[ H_0(X,\Z)_\tau^\wedge \times \Br_1(X) \to \Q/\Z. \]
	In particular, when $X$ is a curve, this pairing gives a topological isomorphism 
	\[ \Br(X)^* \cong H_0(X,\Z)_\tau^\wedge. \]
\end{theorem-no}

When $X$ is a curve, this result extends the classical results of Lichtenbaum since, in this case, $\Br_1(X)$ agrees with $\Br(X)$. A key point to obtain this result is to properly topologize $H_0(X,\Z)_\tau$. At that point, the presence of singularities on $X$ causes some topological difficulties. For instance, when $X$ is not normal, the Cartier dual of its Picard variety is not a algebraic group but a 1-motive (see \S\ref{1-motive section}). Another issue is that some cohomology groups that appear in this setting turn out to be non-Hausdorff (see \S\ref{topologization section}). In our setting, this is not an important issue, but it warns us about the kind of problems that could appear in the case of local fields of positive characteristic. Moreover, with our treatment of the topological issues, we also aim to clarify the proof of van Hamel's generalization of Lichtenbaum's duality. \\

It is important to note that the group $H_0(X,\Z)_\tau$ is hard to describe explicitly because of its derived nature. However, we give an explicit description of this group for a certain kind of varieties. In order to do that, we give an explicit description of the pushforward map $H_0(X',\Z)_\tau\to H_0(X,\Z)_\tau$ induced by the normalization map $X'\to X$ of $X$. Thus, with our duality result, we can describe the map $\Br_1(X)\to\Br_1(X')$ for some singular varieties, which has a particular arithmetic interest. More precisely, the singular varieties treated are obtained by pinching normal proper varieties in a finite collection of closed points.

\section*{Acknowledgments}

I would like to thank my advisors, Cristian Gonz\'alez Avil\'es and Giancarlo Lucchini Arteche, for their constant support while writing this article. I also thank Diego Izquierdo for his carefully reviewing of this article. Finally, I thank {\it Agencia Nacional de Investigaci\'on y Desarrollo} (ANID) for funding my research via Beca de Doctorado Nacional ANID folio 21210171. 

\section{Notations and definitions}

For an abelian category $\s{A}$, we denote $C(\s{A})$ the category of complexes in $\s{A}$ and $D(\s{A})$ its derived category. There is a canonical functor $\s{A}\to D(\s{A})$ that sends an object $A$ in $\s{A}$ to the complex 
	\[ \cdots \to 0 \to 0 \to A \to 0 \to 0 \to \cdots \]
concentrated in degree 0. We still write $A$ to represent this complex. For a right derivable additive functor $F:\s{A}\to\s{B}$ between abelian categories, we denote $RF:D(\s{A})\to D(\s{B})$ its derived functor. Moreover, for $i\in\Z$, we denote $R^iF$ the composition $H^i\circ RF:D(\s{A})\to\s{B}$, where $H^i:D(\s{B})\to\s{B}$ is the $i$th cohomology functor. Recall that when $\s{A}$ has enough injectives, every additive functor $F:\s{A}\to\s{B}$ is right derivable.

For a site $X_E$ over a scheme $X$, we denote $S(X_E)$ the abelian category of sheaves of abelian groups on $X_E$ and a sheaf $\s{F}\in S(X_E)$ is called an \emph{abelian sheaf on} $X_E$. We denote $D(X_E)$ the bounded derived category of $C(S(X_E))$. For a bounded complex of abelian sheaves $\s{F}^\bullet$ on $X_E$, we denote $H^i(X_E,\s{F}^\bullet)$ the abelian group $R^i\Gamma(X_E,\s{F}^\bullet)$. When $X$ is the spectrum of a field $F$, we write $F_E$ instead of $(\Spec F)_E$. Let $G$ be group scheme over $X$. If the functor of points of $G$ defines a abelian sheaf on $X_E$, we also write $G$ to refer to the sheaf $\Hom_X(-,G)$.

Let $k$ be a field. A \emph{$k$-scheme} $X$ is a scheme over $\Spec k$ whose structure morphism is separated and locally of finite type. Moreover, we refer to a $k$-scheme $X$ as a $k$-\emph{variety} if it is geometrically integral. A $k$-variety of dimension one is called a \emph{$k$-curve}. A group scheme over $k$ is called a \emph{(locally) algebraic $k$-group} if its structure map is (locally) of finite type. We denote $\mathcal{C}_k$ the category of commutative algebraic $k$-groups. When $k$ has characteristic zero, every connected commutative algebraic group $G$ is an extension
	\begin{equation} \label{alg group decomp}
		 0 \to T\times \G_a^n \to G \to A \to 0,
	\end{equation}
for some $n\in\N$, where $T$ is a torus and $A$ is an abelian variety (see \cite[Theorems 4.3.2 and 5.3.1]{AG}).

For a scheme $X$, we will denote $X_{\sm}$ the \emph{smooth site} over $X$, i.e. the site whose underlying category is the category of smooth schemes over $X$ and whose coverings are the \emph{jointly surjective families} of smooth maps, i.e. families of smooth maps $\{U_i\to V\}_{i\in I}$ such that $\coprod_{i\in I}U_i\to V$ is surjective.

Throughout this article, $k$ will be a field of characteristic zero and all locally algebraic groups will be commutative. Moreover, we do not assume that a (locally) compact topological group is Hausdorff. In this context, a topological group is \emph{locally compact} if the identity element has a neighborhood basis composed by compact neighborhoods.


\section{Preliminaries}

For the convenience of the reader, in this section we compile well-known results related to algebraic groups. Throughout this section, $k$ will be a field of characteristic zero and all locally algebraic groups will be commutative. 


\subsection{Derived Cartier duality}

Every locally algebraic group $G$ may be canonically regarded as a sheaf on $k_{\fppf}$ by taking its functor of points $G(-)$. Grothendieck proved that this assignment defines canonically a functor $F:\s{C}_k\to S(k_{\fppf})$, which is fully faithful. For our purposes, $k$ has characteristic zero and, therefore, every locally algebraic group $G$ is smooth over $k$. Then, $F$ factors through $\s{C}_k\to S(k_{\sm})$. \\

From now on, we will work over a field $k$. We define the \emph{derived Cartier dual functor} as 
\begin{align} \label{derived Cartier dual}
	\begin{split}
		(-)^{\D} : D(k_{\sm}) &\to D(k_{\sm})  \\
		\s{F} &\mapsto R\sHom_{k_{\sm}}(\s{F},\G_{m,k}),
	\end{split}
\end{align}
where $\sHom_{k_{\sm}}(-,\G_{m,k})$ is the functor that sends a sheaf $\s{F}$ in $S(k_\sm)$ to the sheaf of homomorphism
	\begin{align*}
		U\mapsto \Hom_{k_{\sm}}(\s{F}|_U,\G_{m,U}).
	\end{align*}
Observe that, for any complex $\s{F}$ in $D(k_{\sm})$, the following isomorphisms of abelian groups hold
\begin{equation*} 
	H^i(k_{\sm},\s{F}^\D) \cong \Ext_{k_{\sm}}^i(\s{F},\G_m)
\end{equation*}
for all $i>0$. The derived Cartier dual functor is closely related to the classical Cartier dual functor, defined as 
\[ L^D:=\sHom_{k_{\fppf}}(L,\G_m) \]
for an affine locally algebraic group $L$. We will focus in studying the values of $(-)^\D$ for certain kind of complexes $\s{C}$. Specifically we will be interested in the case when $\s{C}$ is a complex concentrated in degree zero by a sheaf represented by a locally algebraic group. \\

Recall that, when $k$ has characteristic zero, every connected commutative algebraic group $G$ is an extension
\begin{equation} \label{alg group decomp}
	0 \to T\times \G_a^n \to G \to A \to 0,
\end{equation}
for some $n\in\N$, where $T$ is a torus and $A$ is an abelian variety (see \cite[Theorems 4.3.2 and 5.3.1]{AG}). Moreover, when $k$ is perfect, the cohomology group $H^i(k,\G_a^D)$ is trivial for all $i$ (see \cite[Proposition 2.5.1]{Rosengarten}). Hence, 

\begin{proposition} 
	Let $L$ be a commutative linear $k$-group and let $U$ be its unipotent part. Then $H^i(k,L^D)$ is isomorphic as abelian group to $H^i(k,(L/U)^D)$ for all $i$.
\end{proposition}

The following result can be found in \cite[Lemma 1.3.3]{CTSan}.

\begin{lemma} \label{cts}
	Let $G$ be an algebraic $k$-group of multiplicative. The sheaves $\sExt^i_{k_{\fppf}}(G^D,\G_{m,k})$ are trivial for $i>0$. 
\end{lemma}

Thus we can deduce the triviality of the sheaves of extension of finite étale commutative algebraic groups over a field.

\begin{corollary} \label{vanishing ext sheaves}
	Let $F$ be a finite étale commutative algebraic $k$-group. The sheaves $\sExt_{k_{\fppf}}^i(F,\G_m)$ are trivial for all $i>0$.
\end{corollary}

\begin{proof}
	As $F$ is a finite étale group scheme over $k$, we have that $F$ is finite étale locally constant \cite[\href{https://stacks.math.columbia.edu/tag/03RV}{Lemma 03RV}]{stacks-project}. Then $F^D$ is of multiplicative type. Using the natural isomorphism $(F^D)^D\cong F$ and Lemma \ref{cts}, we conclude that the sheaves $\sExt_{k_{\fppf}}^i(F,\G_{m,k})$ are trivial for all $i>0$.
\end{proof}

Thus, the following result follows.

\begin{corollary} \label{finite groups}
	Let $F$ be a finite étale commutative algebraic $k$-group. The following isomorphism holds in $D^b(k_{\fppf})$
	\[ R\sHom_{k_{\fppf}}(F,\G_{m,k}) \cong F^D. \]
\end{corollary}

\begin{proof}
	We only have to show that $\sExt^i_{k_{\fppf}}(F,\G_{m,k})$ are trivial for all $i>0$. Hence we conclude from Corollary \ref{vanishing ext sheaves}.
\end{proof}

Let $G$ be an algebraic group of multiplicative type over $k$. Let $M$ be the finitely generated abelian group associated to $G^D$. Let $M_{\mathrm{tors}}$ be the torsion subgroup of $M$. Then we have an exact sequence of abelian groups
\[ 0 \to M_{\mathrm{tors}} \to M \to M/M_{\mathrm{tors}} \to 0, \]
where $M_{\mathrm{tors}}$ is finite and $M/M_{\tors}$ is isomorphic to $\Z^r$ for some $r\in\Z$. Hence, $G$ can be regarded as an extension 
\[ 0 \to T \to G \to F \to 0, \]
where $F$ is a finite group scheme such that $F^D$ is étale locally isomorphic to $M_{\mathrm{tors}}$ and $T$ is a torus of dimension $r$.

\begin{proposition} \label{dual plano}
	Let $G$ be an algebraic $k$-group of multiplicative type and $A$ an abelian $k$-variety. Then, the following isomorphisms hold in $D^b(k_{\fppf})$:
	\begin{enumerate}[nosep, label=\alph*)]
		\item $\tau_{\leq 1}R\sHom_{k_{\fppf}}(G,\G_{m,k}) \cong G^D$;
		\item $\tau_{\leq 1}R\sHom_{k_{\fppf}}(G^D,\G_{m,k}) \cong G$;
		\item $\tau_{\leq 1}R\sHom_{k_{\fppf}}(A,\G_{m,k}) \cong A^t[-1]$, where $A^t$ is the dual abelian variety of $A$.
	\end{enumerate}
\end{proposition}

\begin{proof}
	\begin{enumerate}[nosep, label=\alph*)]
		\item Let $T$ be the toric part of $G$ and $F$ the finite quotient of $G$ by $T$. We have the following exact triangle 
		\begin{equation} \label{multiplicative}
			T\to G \to F \to T[1]. 
		\end{equation}
		Applying the functor $R\sHom_{k_{\fppf}}(-,\G_{m,k})$ to \eqref{multiplicative}, we get the following exact triangle
		\[ R\sHom_{k_{\fppf}}(F,\G_{m,k}) \to R\sHom_{k_{\fppf}}(G,\G_{m,k}) \to R\sHom_{k_{\fppf}}(T,\G_{m,k}) \to R\sHom_{k_{\fppf}}(F,\G_{m,k})[1]. \]
		As $G$ is smooth, $F$ is étale. Thus, by Corollary \ref{finite groups}, the exact triangle above becomes
		\[ F^D \to R\sHom_{k_{\fppf}}(G,\G_{m,k}) \to R\sHom_{k_{\fppf}}(T,\G_{m,k}) \to F^D[1]. \]
		Applying the truncation $\tau_{\leq 1}$ to the triangle here above yields the following (possibly non-exact) triangle 
		\begin{equation} \label{triangle tori}
			F^D \to \tau_{\leq 1}R\sHom_{k_{\fppf}}(G,\G_{m,k}) \to \tau_{\leq 1}R\sHom_{k_{\fppf}}(T,\G_{m,k}) \to F^D[1].
		\end{equation}
		In fact, triangle \eqref{triangle tori} is exact. In order to prove this assertion, by \cite[Lemma 3.10]{Brochard}, it is sufficient to prove that the sheaf $\sExt_{k_{\fppf}}^1(T,\G_{m,k})$ is trivial. After an étale base change, it is sufficient to prove that the sheaf $\sExt_{k_{\fppf}}^1(\G_{m,k},\G_{m,k})$ is trivial. This last assertion is true by the proof of \cite[Exposé VIII, Proposition 3.3.1]{SGA7I}. Thus, triangle \eqref{triangle tori} turns into
		\[ F^D \to \tau_{\leq 1}R\sHom_{k_{\fppf}}(G,\G_{m,k}) \to T^D \to F^D[1]. \]
		Therefore, $H^1(\tau_{\leq 1}R\sHom_{k_{\fppf}}(G,\G_{m,k}))$ is isomorphic to zero in $D(k_{\fppf})$. Consequently, $\tau_{\leq 1}R\sHom_{k_{\fppf}}(G,\G_{m,k})$ is isomorphic to $G^D$ in $D(k_{\fppf})$.
		
		\item This follows from Lemma \ref{cts}.
		
		\item Since there are no non constant maps from an abelian variety to an affine scheme, the sheaf $\sHom_{k_{\fppf}}(A,\G_{m,k})$ is trivial (see \cite[Exposé VII, \S 1.3.8]{SGA7I}). Finally, by the Barsotti-Weil formula, the sheaf $\sExt^1_{k_{\fppf}}(A,\G_m)$ is naturally isomorphic to the dual abelian variety $A^t$ of $A$ (see, for instance, \cite[Corollary 3.5]{RR}). Whence we conclude that
		\[ \tau_{\leq 1}R\sHom_{k_{\fppf}}(A,\G_{m,k}) \cong \sExt^1_{k_{\fppf}}(A,\G_m)[-1] \cong A^t[-1] \]
		in $D^b(k_{\fppf})$.
	\end{enumerate}
\end{proof}

\begin{lemma} \label{injectivity on ext}
	Let $G$ be a connected algebraic $k$-group and let $n$ be a positive integer. The multiplication by $n$ on $\sExt_{k_{\fppf}}^i(G,\G_{m,k})$ is injective for all $i>1$.
\end{lemma}

\begin{proof}
	Since $k$ has characteristic zero, we may work case by case when $G$ is an abelian variety $A$, a torus $T$ or the additive group $\G_{a,k}$.
	\begin{itemize}
		\item For a $\fppf$ covering $S\to\Spec k$, following isomorphism holds
		\[ \sExt_{k_{\fppf}}^i(T,\G_{m,k})|_S \cong \sExt_{S_{\fppf}}^i(T_S,\G_{m,S}). \]
		Then, as the functor $\sExt_{k_{\fppf}}^i(-,\G_{m,k})$ commutes with finite direct sums, it is enough to check the injectivity of multiplication by $n$ when $T=\G_{m,k}$. We have the Kummer exact sequence
		\begin{equation} \label{Kummer}
			0 \to \mu_n \to \G_{m,k} \overset{(-)^n}{\to} \G_{m,k} \to 0
		\end{equation}
		for all $n\in\N$. Applying the derived functor $R\sHom_{k_{\fppf}}(-,\G_{m,k})$ to \eqref{Kummer}, we get the following exact sequence
		\[ \sExt_{k_{\fppf}}^{i-1}(\mu_n,\G_{m,k}) \to \sExt_{k_{\fppf}}^i(\G_{m,k},\G_{m,k}) \overset{\cdot n}{\to} \sExt_{k_{\fppf}}^i(\G_{m,k},\G_{m,k}), \]
		where the rightside map is multiplication by $n$, which is induced by the $n$-th power map $(-)^n:\G_{m,k}\to\G_{m,k}$. Then, applying the Proposition \ref{finite groups} to the finite commutative group scheme $\mu_n$, we conclude that multiplication by $n$ on  $\sExt_{k_{\fppf}}^i(\G_{m,k},\G_{m,k})$ is injective for all $i>1$. 
		
		\item We know that multiplication by $n$ defines an isogeny from $A$ onto $A$. \cite[\href{https://stacks.math.columbia.edu/tag/03RP}{Proposition 03RP}]{stacks-project}. Thus, we have the following exact sequence
		\begin{equation} \label{n torsion}
			0 \to A[n] \to A \overset{\cdot n}{\to} A \to 0,
		\end{equation}
		where the $n$-torsion $A[n]$ of $A$. Applying the derived functor $R\sHom_{k_{\fppf}}(-,\G_{m,k})$ to \eqref{n torsion}, we get the following exact sequence
		\[ \sExt_{k_{\fppf}}^{i-1}(A[n],\G_{m,k}) \to \sExt_{k_{\fppf}}^i(A,\G_{m,k}) \overset{\cdot n}{\to} \sExt_{k_{\fppf}}^i(A,\G_{m,k}), \]
		where the rightside map is multiplication by $n$, which is induced by multiplication by $n$ in $A$. Since $A[n]$ is finite, we conclude similarly to before above that multiplication by $n$ on $\sExt_{k_{\fppf}}^i(A,\G_{m,k})$ are trivial for all $i>1$.
		
		\item Since $k$ has characteristic zero, multiplication by $n$ defines an isomorphism $\G_a\to\G_a$ for every positive integer $n$. This implies that multiplication by $n$ on $\sExt_{k_{\fppf}}^i(\G_a,\G_{m,k})$ is an isomorphism for every $i\geq 1$.
	\end{itemize}
\end{proof}

The following results ensure the triviality of some sheaves of extensions on the smooth site.

\begin{lemma} \label{triviality of ext}
	Let $L$ be a connected linear algebraic $k$-group and let $A$ be an abelian $k$-variety. The sheaves $\sExt_{k_{\sm}}^i(L,\G_{m,k})$ and $\sExt_{k_{\sm}}^{i+1}(A,\G_{m,k})$ are trivial for all $i\geq 1$.
\end{lemma}

\begin{proof}
	It is sufficient to prove that the groups $\Ext_{X_{\sm}}^i(L,\G_{m,X})$ and $\Ext_{X_{\sm}}^{i+1}(A,\G_{m,X})$ are trivial for every smooth $k$-scheme $X$ and $i\geq 1$.
	Breen has proved that over a regular noetherian scheme $X$, the groups $\Ext_{X_{\sm}}^{i+1}(A,\G_{m,X})$ and $\Ext_{X_{\sm}}^i(L,\G_{m,X})$ are torsion for all $i\geq 1$ when $X=\G_{m,k},\ \G_{a,k}$ or an abelian variety $A$ (see \cite[\S 7 and \S8]{Breen}). Thus, by Lemma \ref{injectivity on ext}, we conclude that $\sExt_{k_{\sm}}^i(L,\G_{m,k})$ and $\sExt_{k_{\sm}}^{i}(A,\G_{m,k})$ are trivial for all $i\geq 2$. Moreover, it also holds that $\sExt_{k_{\sm}}^i(\G_{a,k},\G_{m,k})$ is trivial. Hence, by Proposition \ref{dual plano}.a), we conclude that $\sExt_{k_{\sm}}^i(L,\G_{m,k})$ is trivial for $i=1$.
\end{proof}

\begin{remark}
	In the previous lemma it was crucial that the site in which the extension sheaves was computed is the smooth site. In fact, it is not true that these sheaves are trivial on the flat site. Particularly, the sheaf $\sExt_{k_{\fppf}}^1(\G_{a,k},\G_{m,k})$ is not trivial when $k$ has characteristic zero (see \cite[Remark 2.2.16]{Rosengarten}). Furthermore, by Lemma \ref{triviality of ext}, the sheaf $\sExt_{k_{\fppf}}^1(\G_{a,k},\G_{m,k})$ is uniquely divisible.
\end{remark}

The following result relates the values in certain sheaves of the functors $R\sHom_{k_{\fppf}}(-,\G_{m,k})$ and $R\sHom_{k_{\sm}}(-,\G_{m,k})$.

\begin{lemma} \label{fl sm}
	Let $G$ be a smooth group scheme over $k$. Let $\alpha:k_{\fppf}\to k_{\sm}$ be the canonical map of sites. Then we have a canonical isomorphism
	\[ R\alpha_*R\sHom_{k_{\fppf}}(G,\G_{m,k})\cong R\sHom_{k_{\sm}}(G,\G_{m,k}). \]
\end{lemma}

\begin{proof}
	This is a particular case of \cite[Lemma 1.3]{vH04}.
\end{proof}

\begin{corollary} \label{dual suave}
	Let $L$ be a linear $k$-group and $A$ an abelian $k$-variety. Then, the following isomorphisms hold in $D^b(k_{\sm})$
	\begin{enumerate}[nosep, label=\alph*)]
		\item $L^\D \cong L^D$;
		\item $A^\D \cong A^t[-1]$, the dual abelian variety of $A$.
	\end{enumerate}
	Furthermore, for a connected algebraic $k$-group $G$, which fits in a extension
	\begin{equation} \label{algebraic group extension}
		0 \to L \to G \to A \to 0,
	\end{equation}
	the following isomorphism holds in $D^b(k_{\sm})$
	\begin{enumerate}[nosep, label=\alph*), start=3]
		\item $G^\D \cong [L^D\overset{\partial}{\to} A^t][-1]$,
	\end{enumerate}
	where $\partial:\sHom_{k_{\sm}}(L,\G_m)\to\sExt^1_{k_{\sm}}(A,\G_m)$ is the connecting map associated to the exact sequence \eqref{algebraic group extension} once applied the derived Cartier dual functor $(-)^\D$.
\end{corollary}

\begin{proof}
	In any case, all groups schemes involved are smooth and locally of finite type over $k$. Then, by vanishing of higher direct images of sheaves represented by smooth group schemes (cf. \cite[p. 114, Chapter III, Theorem 3.9]{EC}), the result follows from Proposition \ref{dual plano}, Lemma \ref{triviality of ext} and Lemma \ref{fl sm}.
\end{proof}


\subsection{Generalized 1-motives} \label{1-motive section}

Throughout this section, $k$ will be a field of characteristic zero and all locally algebraic groups will be commutative.  A \emph{generalized 1-motive} over $k$ is a complex $M=[L^D\to G]$ in $C^b(S(k_{\sm}))$ concentrated in degrees $-1$ and 0, where $L$ is a connected linear group and $G$ is a connected algebraic group. A morphism of generalized 1-motives is just a morphism of complexes. When $L$ is a torus and $G$ is a semiabelian variety, we are in the case of a Deligne 1-motive.

Let $M=[L^D\to G]$ be a generalized 1-motive. Using the decomposition $L=T\times\G_a^n$, for some $n\in\N$, as in \eqref{alg group decomp}, $M$ fits into an exact sequence of complexes
\[ 0 \to (\G_ a^D)^n[1] \to M \to M' \to 0, \]
where $M'$ is the generalized 1-motive $[T^D\to G]$. Hence, the map $M\to M'$ above induces a canonical isomorphism 
\begin{equation} \label{motives withno ga}
	H^i(k,M)\to H^i(k,M')
\end{equation} 
for all $i$. Indeed, $H^i(k,\G_a^D)=0$ for all $i$ (see \cite[Proposition 2.5.1]{Rosengarten}).

Let $G$ be a connected algebraic $k$-group. Recall that the derived Cartier dual $G^\D$ of $G$ is the complex $R\sHom_{k_{\sm}}(G,\G_m)$ (see \eqref{derived Cartier dual}). By Corollary \ref{dual suave}.c), the complex $G^\D[1]$ is represented by a generalized 1-motive. \\

Thus, with a connected algebraic group we can associate a Deligne 1-motive as follows.

\begin{proposition} \label{Deligne 1-motive associated}
	Let $G$ be a connected algebraic group over $k$. There exists a Deligne 1-motive $M$ and an exact triangle
	\[ (\G_a^D)^n \to G^\D \to M[-1] \to (\G_a^D)^n[1] \]
	that induces isomorphism of groups $H^i(k,G^\D) \cong H^{i-1}(k,M)$ for all $i$. 
\end{proposition}


\subsection{Yoneda pairing}

Let $G$ be a locally algebraic group over $k$. There exists a natural pairing in $D(k_{\sm})$
	\[ G \otimes^{\mathbf{L}} G^\D \to \G_m \]
that induces the following pairing of abelian groups
	\[ H^i(k,G) \times H^{2-i}(k,G^\D) \to H^2(k,\G_m). \]
Observe that this pairing is functorial on $G$. This pairing is crucial in Tate local duality theorems (see for instance Proposition \ref{tate}).


\subsection{Basics on topological groups}

A continuous homomorphism of topological groups $f:A\to B$ is called \emph{strict}\footnote{In \cite{Stroppel}, the author named a strict morphism as \emph{proper}.} if the image of every open subset is open in the image $f(A)$ with the subspace topology in $B$, or equivalently, if the map $\tilde{f}:A/\ker f\to f(A)$ is an isomorphism of topological groups, where $A/\ker f$ has the quotient topology (see \cite[III, \S 2.8, p. 236, Proposition 24]{GTBou}).

An exact sequence of abelian topological groups
\[ \cdots \to A_{i-1} \overset{f_{i-1}}{\to} A_i \overset{f_i}{\to} A_{i+1} \to \cdots \]
is called \emph{strict exact} if every map $f_i$ is continuous and strict. A strict exact sequence of the form
\[ 0 \to A \to B \to C \to 0 \]
is called \emph{a topological extension}.

An abelian topological group is called \emph{profinite} if it is isomorphic to an inverse limit of discrete finite groups. Being profinite is equivalent to being Hausdorff, compact and totally disconected (see \cite[p. 172, Theorem 19.9]{Stroppel}). We can assign a profinite abelian group to every abelian topological group. Let $A$ be an abelian topological group. We define the \emph{profinite completion of $A$} as the inverse limit
\[ A^\wedge := \varprojlim_H A/H \leq \prod_H A/H \]
running over all open subgroups of finite index of $A$, equipped with the subspace topology inherehited from the product $\prod_H A/H$, where $A/H$ has discrete topology. Thus, $A^\wedge$ is a closed subgroup of $\prod_H A/H$. On the other hand, there is a canonical continuous morphism of topological abelian groups
\begin{equation} \label{completion map} 
	c_A : A \to A^{\wedge},
\end{equation}
which has dense image (see \cite[p. 167, Lemma 18.4]{Stroppel}). The profinite completion of $A$ is the universal object for maps from $A$ to profinite abelian groups, that is, for every continuous morphism $f:A\to P$, where $P$ is a profinite abelian group, there exists a unique continuous morphism $\hat{f}:A^\wedge\to P$ such that $f=\hat{f}\circ c_A$ (see \cite[p. 166, Lemma 18.1]{Stroppel}). Every profinite abelian group agrees with its profinite completion. For every continuous morphism of abelian topological groups $f:A\to B$, there is a canonical continuous morphism
\[ f^\wedge:A^\wedge \to B^\wedge, \]
which is also strict (see \cite[III, \S 2.8, p. 237, Remark 1]{GTBou}). Thus, $(-)^\wedge$ defines a functor from the category of abelian topological groups to the category of profinite abelian groups. This functor is right exact in the following sense: for every topological extension of abelian topological groups
\begin{equation*}
	0\to A \to B \to C \to 0,
\end{equation*}
the sequence of abelian topological groups
\begin{equation*}
	A^\wedge \to B^\wedge \to C^\wedge \to 0
\end{equation*}
is strict exact \cite[Proposition C.1.4]{Rosengarten}.

\begin{remark} \label{exactness profinite}
	The profinite completion functor is not left exact in general. For instance, the canonical embedding $i:\Z\to\Q$ of discrete abelian groups induces the zero map $i^\wedge:\hat{\Z}\to 0$ since the profinite completion of the discrete group $\Q$ is trivial. At any rate, a topological embedding $A\to B$ will induce a topological embedding $A^\wedge\to B^\wedge$ in each of the following cases :
	\begin{enumerate}[nosep, label=(\alph*)]
		\item The image of $A\to B$ is closed in $B$ and $B$ is Hausdorff, locally compact, totally disconnected and compactly generated\footnote{An abelian topological group $G$ is compactly generated if there exists a compact subset $S\subseteq G$ such that $\langle S\rangle=G$.} (see \cite[Appendix, Proposition]{HS05}).
		\item $B/A$ is discrete and finitely generated (see \cite[Proposition C.2.1]{Rosengarten}).
	\end{enumerate}
\end{remark}

The following lemma is obvious but will be often used later.

\begin{lemma} \label{profinite}
	Let $\phi:A\to B$ be a continuous homomorphism of abelian topological groups. Assume that $B$ is Hausdorff and $\phi$ is strict with finite image. Then there exists a continuous and strict homomorphism $\hat{\phi}:A^\wedge \to B$ such that $\phi=\hat{\phi}\circ c_A$ and $\im\hat{\phi}=\im\phi$.
\end{lemma}

\begin{proof}
	Since $\phi$ has finite image and $B$ is Hausdorff, we have that $\ker\phi$ is a closed subgroup of $A$ with finite index. Then, $\ker\phi$ is an open subgroup of $A$ with finite index. Thus we define $\hat{\phi}:A^\wedge \to B$ as the composition
	\[ A^\wedge \overset{\pi_{\ker\phi}}{\longrightarrow} A/\ker\phi \overset{\phi'}{\longrightarrow} B, \]
	where $\pi_{\ker\phi}$ is the canonical projection and $\phi'$ is the composition $A/\ker\phi\overset{\sim}{\to}\im\phi\to B$. Since $B$ is Hausdorff, the inclusion $\im\phi\to B$ is strict. Hence, $\hat{\phi}$ is continuous and strict because it is the composition of a topological embedding and a quotient map (see \cite[III, \S 2.8, p. 237, Remark 2]{GTBou}). Moreover, clearly, the equality $\phi=\hat{\phi}\circ c_A$ holds.
\end{proof}

Let $A$ and $B$ be two abelian topological groups. We denote $\Hom_{\cts}(A,B)$ the group of continuous homomorphisms from $A$ to $B$. The group $\Hom_{\cts}(A,B)$ can be endowed with a natural topology as follows: given a compact subset $K$ of $A$ and an open subset $U$ of $B$, we define $V(K,U)$ as
\[ V(K,U) := \{f:A\to B\text{ continuous}\mid f(K)\subseteq U\}. \]
The family $\{V(K,U)\}_{K,U}$ generates a topology on the group of all continuous maps $C(A,B)$ from $A$ to $B$, which is called the \emph{compact-open topology on} $C(A,B)$. Thus, $\Hom_{\cts}(A,B)$ is endowed with the subspace topology inherehited from $C(A,B)$.

Let $\mathbb{T}$ be the unit circle $\R/\Z$ with the quotient topology induced by $\R$ with its usual metric topology. The torsion subgroup of $\mathbb{T}$ is $\Q/\Z$.

\begin{definition}
	For an abelian topological group $A$, we define its \emph{Pontryagin dual} as the abelian topological group $A^*:=\Hom_{\cts}(A,\mathbb{T})$.
\end{definition}

Observe that $A^*$ is always Hausdorff, even if $A$ is not. Recall that a topological group is \emph{locally compact} if the identity element has a neighborhood basis composed by compact neighborhoods.

\begin{fact} \label{Pontryagin dual}
	Let $A$ be an abelian topological group. Then:
	\begin{enumerate}[nosep, label=\alph*)]
		\item If $A$ is locally compact then so is $A^*$ (see \cite[p. 175, Theorem 20.5]{Stroppel});
		\item If $A$ is compact, then $A^*$ is discrete; and if $A$ is discrete, then $A^*$ is Hausdorff compact (see \cite[p. 175, Theorem 20.6]{Stroppel});
		\item If $A$ is torsion or profinite, then $A^*$ agrees with the group $\Hom_{\cts}(A,\Q/\Z)$ (see \cite[p. 60, Lemma 2.9.2]{Profinite}). 
		\item If $A$ is first countable then so is $A^*$.
	\end{enumerate}
\end{fact}

In the following, we recall the main result related to the Pontryagin dual. Recall that $\LCA$ is the category of Hausdorff and locally compact abelian topological groups with continuous homomorphism as morphisms.

\begin{proposition}[Pontryagin duality] \label{Pontryagin duality}
	The functor $(-)^*\circ(-)^*:\LCA\to\LCA$ is naturally isomorphic to the identity functor $\id_{\LCA}$. Morever, the Pontryagin dual functor $(-)^*:\LCA\to\LCA$ establishes an anti-equivalence between the category of discrete abelian torsion groups and the category of profinite abelian groups.
\end{proposition}

\begin{proof}
	For the first assertion see \cite[p. 193, Theorem 22.6]{Stroppel} and for the last assertion see \cite[p. 62, Theorem 2.9.6]{Profinite}.
\end{proof}

Another very important result is that $(-)^*$ preserves exactness in $\mathbf{LCA}$.

\begin{proposition} \label{pontryagin exactness}
	Let $\s{E}:0 \to A \to B \to C \to 0$ be a topological extension in $\mathbf{LCA}$. Then the dual sequence
	\[ \s{E}^*:0\to C^* \to B^* \to A^* \to 0 \]
	is a topological extension in $\mathbf{LCA}$.
\end{proposition}

\begin{proof}
	See \cite[p. 195, Corollary 23.4]{Stroppel}.
\end{proof}

Let $A$ and $B$ be two abelian topological groups. A pairing $A\times B\to\Q/\Z$ is a bi-additive map. We say that a pairing $A\times B\to\Q/\Z$ is continuous if it is a continuous map between $A\times B$ with the product topology and $\Q/\Z$ with the discrete topology. Further, a continuous pairing $A\times B\to\Q/\Z$ is \emph{perfect} if the induced maps of abelian topological groups $A\to\Hom_{\cts}(B,\Q/\Z)$ and $B\to\Hom_{\cts}(A,\Q/\Z)$ are topological isomorphisms. The following result gives a way to prove when a pairing $A\times B\to\Q/\Z$ is continuous.

\begin{proposition} \label{pairing continuous}
	Let $\pint{-}{-}:A\times B\to\Q/\Z$ be a pairing between locally compact abelian topological groups. Assume that
	\begin{enumerate}[nosep, label=\roman*)]
		\item $\pint{-}{-}$ is continuous at $(0_A,0_B)$;
		\item $\pint{-}{b}:A\to\Q/\Z$ is continuous at $0_A$ for every $b\in B$; and
		\item the map $B\to\Hom_{\cts}(A,\Q/\Z):b\mapsto\pint{-}{b}$ is continuous at $0_B$. 
	\end{enumerate}
	Then $\pint{-}{-}$ is continuous.
\end{proposition}

\begin{proof}
	See \cite[Proposition 1.12]{GAlocal}.
\end{proof}

\begin{remark} \label{pairing discrete}
	If $A$ in Proposition \ref{pairing continuous} is discrete, we only have to check $iii)$ in order to establish the continuity of $\pint{-}{-}$. Indeed, $ii)$ holds trivially and $i)$ holds since $\{0_A\}$ is an open neighborhood of $0_A$ and for every open neighborhood $V$ of $0_B$ we have that $\pint{0_A}{V}=0$. Hence $\pint{-}{-}$ is continuous at $(0_A,0_B)$.
\end{remark}

Recall that, for a topological group $G$, its \emph{separation} or \emph{Hausdorff quotient} $G_\Haus$ is defined as the quotient group $G/\overline{\{0_G\}}$. Note that, for every topological group $A$, the topological group $\Hom_{\cts}(A,\Q/\Z)$ is Hausdorff. Then, a continuous pairing of locally compact abelian topological groups $\pint{-}{-}:A\times B\to\Q/\Z$ induces a continuous pairing $\pint{-}{-}_{\Haus}:A_{\Haus}\times B_{\Haus}\to\Q/\Z$.


\subsection{Topology on cohomology groups over a $p$-adic field} \label{topologization section}

\paragraph{Locally algebraic groups} 

Since all group schemes considered are smooth, the cohomology groups $H^n(k,G)$ can be indistinctly regarded as groups of Galois cohomology, étale cohomology, smooth cohomology or flat cohomology. For this reason, we do not make distinction in the notation. \\

Let $G$ be a locally algebraic group over a $p$-field $k$. There is an immersion $G(k)\to\P_k^n(k)$. Now, $\P_k^n(k)$ is naturally endowed with a topology comming from $k$. Thus, $G(k)$ is equipped with the subpace topology from $\P_k^n(k)$. Moreover, we will assume that all cohomology groups $H^n(k,G)$, for $n\geq 1$, has discrete topology. \\

The following result compiles the topological properties of the cohomology groups of locally algebraic groups.

\begin{proposition} \label{top alg grp}
	Let $G$ be a locally algebraic group over a $p$-adic field $k$. Then:
	\begin{enumerate}[nosep, label=\alph*)]
		\item The group $G(k)$ is Hausdorff, locally compact and totally disconnected. Moreover, $G(k)$ is second countable whenever $G$ is algebraic; and if $G$ is also proper, then $G(k)$ is compact and therefore profinite. When $G$ is an étale group scheme, then $G(k)$ is discrete; 
		\item The group $H^n(k,G)$ is (discrete) torsion for all $n\geq 1$;
		\item The group $H^n(k,G)$ is trivial for $n\geq 3$.
	\end{enumerate}
\end{proposition}

\begin{proof}
	For $a)$, the first three properties follow respectively from \cite[Propositions 2.17 and 3.7]{C15} and \cite[Proposition 5.4]{Conrad}\footnote{This reference contains the three properties of this assertion. However we opt for refering to \cite{C15} because in this article the continuity of the maps in the long exact sequence of cohomology groups associated to short exact sequences of commutative locally algebraic groups is treated.}. The second countability when $G$ is algebraic follows from the fact that $G(k)$ is a subspace of the second countable space $\P^n_k(k)$ for some $n\in\N$ \cite[p. 24, Proposition 2.A.12.(2)]{Cornulier}. The assertion about compactness in the case of $G$ being proper follows from \cite[Lemma 2.12.1.(b)]{C15}. For the assertion concerning to étale group schemes, see \cite[Lemma 5.3]{Conrad}. For $b)$ see \cite[p. 73, Corollary 4.23]{GC}. The point $c)$ follows from the fact that $k$ has strict cohomological dimension 2 \cite[p. 134, Theorem 10.6]{GC}.
\end{proof}

\begin{remark}
	Recall that for a connected linear $k$-group $L$ over a $p$-adic field $k$, its group of $k$-points $L(k)$ has the property that all its finite-index subgroups are open. Indeed, $L(k)=T(k)\times\G_a^n(k)$, for some torus $T$ and $n\geq 0$, and both $T(k)$ and $\G_a^n(k)$ have the property of finite-index subgroups being open. Therefore, for any connected algebraic $k$-group $G$, its group of $k$-groups $G(k)$ has the property of all its finite-index subgroups being open (cf. Example 3.11 in \cite{topart}).
\end{remark}

Furthermore, the assignment of the topology in Proposition \ref{top alg grp} behaves well with respect to long exact sequences in cohomology associated to short exact sequences of locally algebraic groups.

\begin{proposition} \label{prop les}
	Let $0 \to H \to G \to Q \to 0$ be an exact sequence of locally algebraic groups over a $p$-adic $k$. Then all the maps in the long exact sequence of cohomology groups 
	\begin{equation} \label{les}
		0 \to H(k) \to G(k) \to Q(k) \to H^1(k,H) \to H^1(k,G) \to H^1(k,Q) \to \cdots
	\end{equation}
	are continuous. Moreover $H(k) \to G(k)$ is closed and $G(k) \to Q(k)$ is open. Hence, the sequence \eqref{les} is strict exact.  
\end{proposition}

\begin{proof}
	The continuity of all the maps in \eqref{les} follows from \cite[Proposition 4.2]{C15}. Then $H(k) \to G(k)$ is closed since $Q(k)$ is Hausdorff. The fact that $G(k) \to Q(k)$ is open follows from \cite[Proposition 4.3.(a)]{C15}. Hence, since all higher cohomology groups involved have the discrete topology, we conclude that all continuous maps in \eqref{les} are strict.
\end{proof}

Let $A$ be an abelian variety and $N$ be an algebraic group of multiplicative type. Then, by Proposition \ref{top alg grp}, $A(k)$ is profinite (see also \cite[Chapter I, p. 41, Lemma 3.3]{ADT}) since $A$ is proper over $k$. Moreover, $N^D(k)$ is discrete. However, $N(k)$ is neither discrete nor profinite, even when $N$ is a torus. Nevertheless, every subgroup of $N(k)$ of finite index is open when $k$ is a $p$-adic field (see \cite[p. 26]{ADT}). This plays an important role in duality results since, in order to obtain the perfection of certain pairings, it is necessary to take profinite completion. 

Recall that an abelian torsion group is \emph{of cofinite type} if its $n$-torsion subgroups are finite for all $n\in\N$.

\begin{proposition}
	Let $A$ be an abelian variety and $N$ an algebraic group of multiplicative type over a $p$-adic field $k$. Then,
	\begin{enumerate}[nosep, label=\alph*)]
		\item $H^1(k,A)$ is torsion of cofinite type and $H^2(k,A)$ is trivial; and
		\item $H^1(k,N)$ and $H^1(k,N^D)$ are finite.
	\end{enumerate}
\end{proposition}

\begin{proof}
	For $a)$ see \cite[Chapter I, p. 43, Corollary 3.4]{ADT}. For $b)$ see \cite[Chapter I, p. 28, Corollary 2.4]{ADT} and \cite[Chapter I, p. 26, Theorem 2.1]{ADT} respectively.
\end{proof}


\paragraph{Generalized 1-motives}

Now we treat the case of generalized 1-motives over $k$, that is, we treat the way of topologizing the cohomology groups $H^n(k,M)$ of a complex $M=[L^D\to G]$ located at degrees $-1$ and 0 where $L$ is a connected linear group and $G$ is a connected algebraic group. The following exact triangle
\[ L^D \to G \to M \to L^D[1], \]
yields the following exact sequences of cohomology groups
\begin{equation} \label{nagao 1-motive 1}
	H^n(k,G) \overset{i_n}{\to} H^n(k,M) \overset{\partial_n}{\to} H^{n+1}(k,L^D)
\end{equation}
for every $n\geq -1$. Thus, for $n\geq 0$, $H^n(k,M)$ is the central term of an exact sequence of abelian groups
\begin{equation} \label{nagao 1-motive 2}
	0 \to H^n(k,G)/\ker i_n \overset{i_n'}{\to} H^n(k,M) \overset{\partial_n'}{\to} \im\partial_n \to 0,
\end{equation}
where $H^n(k,G)/\ker i_n$ is naturally endowed with the quotient topology from $H^n(k,G)$ and $\im\partial_n$ is naturally endowed with the subspace topology from $H^{n+1}(k,L^D)$. Recall that $H^n(k,G)$ and $H^{n+1}(k,L^D)$ are first countable abelian topological groups (see Proposition \ref{top alg grp}). Then, $H^n(k,G)/\ker i_n$ is also first countable since the quotient map $H^n(k,G)\to H^n(k,G)/\ker i_n$ is open. Moreover, $\im\partial_n$ is discrete for all $n\geq -1$. Hence, for every $n\geq -1$, there exists a unique topology on $H^n(k,M)$ such that \eqref{nagao 1-motive 2} becomes a topological extension (see \cite[Proposition 4.1, Corollary 2]{Calabi}). Thus, we get the following result:

\begin{proposition} \label{top motives}
	Let $M$ be a generalized 1-motive $[L^D\to G]$. Then, for every $n\in\Z$, there exists a unique topology on $H^n(k,M)$ such that the sequence \eqref{nagao 1-motive 1} is a strict exact sequence of abelian topological groups. More precisely, the topologies are the following:
	\begin{enumerate}[nosep, label=\alph*)]
		\item The group $H^0(k,M)$ is an extension of a discrete finite group by a quotient of $G(k)$. In particular, $H^0(k,M)$ satisfies that all its subgroups of finite index are open;
		\item For $n\geq 1$ or $n=-1$, $H^n(k,M)$ is endowed with the discrete topology;
		\item For $n< -1$, all the groups $H^n(k,M)$ are trivial.
	\end{enumerate}
	In particular, these groups are locally compact and second countable; and Hausdorff totally disconnected for $n\neq 0$. Moreover, the group $H^0(k,M)_{\Haus}$ is totally disconnected.
\end{proposition}

\begin{proof}
	We have already proved the existence and uniqueness of the topology. Note that $H^1(k,L^D)$ is finite and that all finite-index subgroups of $G(k)$ are open. So assertion $a)$ follows from the previous construction and \cite[Proposition 3.10]{topart}. Points $b)$ and $c)$ follow directly from the construction. For the last assertion: local compacity follows from \cite[p. 54, Theorem 6.7.(c)]{Stroppel}; second countability follows from \cite[p. 36, Proposition 2.C.8.(1),(2)]{Cornulier}; and totally disconnection follows from \cite[p. 56, Proposition 6.9.(b),(c)]{Stroppel} and \cite[Lemma 2.13]{topart}.
\end{proof}

\begin{remarks} \label{h0 motives}
	Let $M=[L^D\to G]$ be a generalized 1-motive. 
	\begin{enumerate}[nosep,label=\alph*)]
		\item The group $H^0(k,M)$ may be not Hausdorff since $\ker i_0$ may not be closed in $G(k)$.
		\item Observe that, with this topology, the isomorphism defined in \eqref{motives withno ga} becomes a topological isomorphism.
		\item When $G$ is an abelian variety, the profinite completion $H^0(k,M)^\wedge$ and the separation $H^0(k,M)_{\Haus}$ of $H^0(k,M)$ agree. Indeed, in this case, $H^0(k,M)$ is compact and, therefore, $H^0(k,M)_{\Haus}$ is Hausdorff compact totally disconnected, that is profinite.
	\end{enumerate}
\end{remarks}

It is important to note that the assignment of topologies in Proposition \ref{top motives} agrees with the one given by Harari and Szamuely in \cite[\S 2]{HS05}. In particular, the cohomology groups $H^n(k,[0\to G])$ and $H^n(k,G)$ are canonically isomorphic as topological groups. This fact is important because we will need the duality results for Deligne 1-motives over $p$-adic fields proven by these authors. In these results, for a Deligne 1-motive $M$, they introduced a modified profinite completion for the group $H^{-1}(k,M)$, which can be done for a generalized 1-motive $M=[L^D\to G]$ as follows: we defined $H^{-1}_\wedge(k,M)$ as
	\[ H^{-1}_\wedge(k,M) := \ker(L^D(k)^\wedge \to G(k)^\wedge), \] 
which is certainly profinite viewed with the subspace topology from $L^D(k)^\wedge$. Harari and Szamuely proved that, when $M$ is a Deligne 1-motive, the Yoneda pairing
	\[ H^{-1}(k,M) \times H^{3}(k,M^{\D}) \to \Q/\Z \]
induces a perfect continuous pairing
	\[ H^{-1}_\wedge(k,M) \times H^{3}(k,M^{\D}) \to \Q/\Z. \]
Now, in general, for a generalized 1-motive $M=[L^D\to G]$, there exists a canonical topological embedding 
	\[ \alpha^\wedge: H^{-1}(k,M)^\wedge \to H^{-1}_\wedge(k,M), \]
where $\alpha:H^{-1}(k,M) \to H^{-1}_\wedge(k,M)$ is the canonical map induced from $c_{L^D(k)}$ given in \eqref{completion map}. Indeed, the exact sequence of abelian discrete finitely generated groups
	\[ 0 \to H^{-1}(k,M) \overset{i}{\to} L^D(k) \to \coker i \to 0 \]
induces the following strict exact sequence of profinite groups
	\[ 0 \to H^{-1}(k,M)^\wedge \overset{i^\wedge}{\to} L^D(k)^\wedge \to (\coker i)^\wedge \to 0 \]
(see \cite[Proposition C.2.1]{Rosengarten}). Thus, the canonical injection $\gamma:\coker i\to G(k)$ yields a factorisation of the continuous morphism $L^D(k)^\wedge\to G(k)^\wedge$ as $L^D(k)^\wedge\to (\coker i)^\wedge \overset{\gamma^\wedge}{\to} G(k)^\wedge$. Observe that all these continuous morphism are strict since they are continuous morphisms between profinite groups. Applying the $\ker$-$\coker$ sequence (see \cite[1.2 Product-Lemma]{kercoker}) to this composition, we obtain the following strict exact sequence of profinite groups
	\[ 0 \to H^{-1}(k,M)^\wedge \overset{\alpha^\wedge}{\to} H^{-1}_\wedge(k,M) \to \ker\gamma^\wedge \to 0. \] 
In general, $\ker\gamma^\wedge$ is not trivial, therefore, $\alpha^\wedge:H^{-1}_\wedge(k,M)\to H^{-1}(k,M)^\wedge$ may not be an isomorphism. However, when $G$ is an abelian variety, we can prove that $\alpha^\wedge$ is a topological isomorphism.

\begin{proposition} \label{h-1}
 Let $M:=[L^D\overset{u}{\to}A]$ be a generalized 1-motive, where $A$ is an abelian variety. Then, $\alpha^\wedge$ is a topological isomorphism.
\end{proposition} 

\begin{proof}
	We only have to prove that $\ker \gamma^\wedge$ defined above is trivial since $\alpha^\wedge$ is a continuous map between profinite groups (see \cite[III, \S2.8, p. 237, Remark 1]{GTBou}). As abelian groups, $\coker i$ is isomorphic to $\im u^0$. Then, since $A(k)$ is profinite, we have the following commutative triangle of abelian topological groups
	\[ \begin{tikzcd}
		\im u^0 \ar{r} \ar{d}{\iota} & A(k) \\
		(\im u^0)^\wedge \ar{ru} & ,
	\end{tikzcd}
	\]
where $\im u^0$ is regarded with the subspace topology from $A(k)$. Then, by \cite[Lemma 1.7]{GAlocal}, $(\im u^0)^\wedge$ agrees with the closure of $\im u^0$ in $A(k)$. Note that, $\gamma^\wedge$ is injective if and only if $(\im u^0)^\wedge\to A(k)$ is so. Hence, $\gamma^\wedge$ is injective and, therefore, $\alpha$ is a topological isomorphism.
\end{proof}

\paragraph{Cartier dual of locally algebraic groups}

Let $G$ be a locally algebraic group over $k$. We have the following exact sequence
	\[ 0 \to G^0 \to G \to \pi_0(G) \to 0, \]
where $G^0$ is the connected component of the identity of $G$ and $\pi_0(G)$ is the group of connected components, which we assume to be finitely generated. Applying the derived Cartier dual functor to the exact sequence above yields the following exact triangle in $D^b(k_{\sm})$
	\[ \pi_0(G)^D \to G^\D \to [L^D\to A^t][-1] \to \pi_0(G)^D[1], \]
where $L$ is the linear part of $G^0$ and $A$ is its abelian quotient (see Corollary \ref{dual suave}). Thus, we obtain the following exact sequence of cohomology groups
	\begin{equation} \label{dual locally}
	H^n(k,\pi_0(G)^D) \overset{a_n}{\to} H^n(k,G^{\D}) \overset{b_n}{\to} H^{n-1}(k,[L^D\to A^t]).
	\end{equation}
By Proposition \ref{top motives}, there exists a unique topology on $H^n(k,G^{\D})$ such that \eqref{dual locally} becomes strict exact for $n\neq 1$. Moreover, by Proposition \ref{top alg grp}, this topology is the discrete topology for $n\neq 0,1$. For $n=1$, the situation is a bit complicated. In this case, if we could equip $H^1(k,G^{\D})$ with a topology such that \eqref{dual locally} becomes a strict exact sequence of topological groups, then the separation $H^1(k,G^{\D})_{\Haus}$ would be canonically determined. That is, if $\tau$ and $\tau'$ are two topologies on $H^1(k,G^{\D})$ such that \eqref{dual locally} becomes a strict exact sequence of topological groups, then $(H^1(k,G^{\D}),\tau)_{\Haus}$ and $(H^1(k,G^{\D}),\tau')_{\Haus}$ are canonically isomorphic, as we explain now. \\

\begin{proposition} \label{top loc alg}
	Let $G$ be a locally algebraic group over a $p$-adic field $k$. There exists a topology $\tau$ on $H^1(k,G^\D)$ such that the Yoneda pairing
	\begin{equation} \label{yoneda paring loc alg grp} 
		(H^1(k,G^\D),\tau) \times H^{1}(k,G) \to H^2(k,\G_m) \cong \Q/\Z
	\end{equation}
	is continuous and the induced map $(H^1(k,G^\D),\tau) \to H^{1}(k,G)^*$ is continuous and strict. Moreover, the induced pairing
	\[ (H^1(k,G^\D),\tau)^\wedge \times H^{1}(k,G) \to \Q/\Z \]
	is continuous and perfect. Furthermore, for any other topology $\tau'$ on $H^1(k,G^\D)$ satisfying the property above, there exists a topological isomorphism $(H^1(k,G^\D),\tau)^\wedge \cong (H^1(k,G^\D),\tau')^\wedge$.
\end{proposition}

Recall that for an abelian topological group $H$, the separation map $q_H:H\to H_{\Haus}$ induces an isomorphism $q_H^\wedge:H^\wedge\to (H_{\Haus})^\wedge$ between their profinite completions (see \cite[Lemma 1.9.(ii)]{GAlocal}). And we will prove here below that $(H^1(k,G^{\D}),\tau)_\Haus$ is profinite and, therefore, $(H^1(k,G^{\D}),\tau)_\Haus \cong (H^1(k,G^{\D}),\tau)^\wedge$.

\begin{proof}
The Yoneda pairing
	\[ H^i(k,G^\D) \times H^{2-i}(k,G) \to H^2(k,\G_m) \cong \Q/\Z \]
yields the following commutative diagram of abelian groups
	\begin{equation*}
	\begin{tikzcd}[column sep=small]
	H^{-1}(k,[L^D\to A^t]) \ar[r]  \ar[d,"\alpha"] & H^1(k,\pi_0(G)^D) \ar{r}{a_1} \ar[d,"\beta"] & H^1(k,G^{\D}) \ar{r}{b_1} \ar{d}{\gamma} & H^0(k,[L^D\to A^t]) \ar[r] \ar{d}{\delta} & H^2(k,\pi_0(G)^D) \ar[d,"\epsilon"] \\
	H^{2}(k,G^0)^* \ar[r] & H^1(k,\pi_0(G))^* \ar{r}{a_2} & H^1(k,G)^* \ar{r}{b_2} & H^{1}(k,G^0)^* \ar[r] & H^0(k,\pi_0(G))^*.
	\end{tikzcd}
	\end{equation*}
The top row is an exact sequence of abelian groups, where all groups except the middle one are topological groups. The bottom row is a topological exact sequence of Hausdorff abelian topological groups (see Propositions \ref{pontryagin exactness} and \ref{prop les})). Since $\pi_0(G)$ is étale locally isomorphic to a constant finitely generated group, $\beta$ and $\epsilon$ are topological isomorphisms by the classical local Tate duality (see \cite[II, \S 5.8]{SerreCG}). Since $H^1(k,\pi_0(X)^D)$ is finite and discrete and $H^2(k,G^0)^*$ is profinite (see Fact \ref{Pontryagin dual},b) and c) and Propositions \ref{Pontryagin duality} and \ref{top alg grp}), we may replace $H^{-1}(k,[L^D\to A^t])$ by its profinite completion without disrupting neither the commutativity nor the exactness of the diagram above. Thus, $\alpha^\wedge$ is a topological isomorphism (see \cite[Theorem 2.3]{HS05} and Proposition \ref{h-1}). By the same result, $\delta$ is a quotient map (see Remark \ref{h0 motives}.c)) and $\ker\delta$ agrees with the indiscrete group $\overline{\{0\}}$. Thus, we obtain the following diagram of abelian groups
	\begin{equation*}
	\begin{tikzcd}[column sep=small]
	H^{-1}(k,[L^D\to A^t])^\wedge \ar[r] \ar[d,"\alpha^\wedge","\vsim"'] & H^1(k,\pi_0(G)^D) \ar{r}{a_1} \ar[d,"\beta","\vsim"'] & H^1(k,G^{\D}) \ar{r}{b_1} \ar[d,twoheadrightarrow,"\gamma"] & H^0(k,[L^D\to A^t]) \ar[r] \ar[d,twoheadrightarrow,"\delta"] & H^2(k,\pi_0(G)^D) \ar[d,"\epsilon","\vsim"'] \\
	H^{2}(k,G^0)^* \ar[r] & H^1(k,\pi_0(G))^* \ar{r}{a_2} & H^1(k,G)^* \ar{r}{b_2} & H^{1}(k,G^0)^* \ar[r] & H^0(k,\pi_0(G))^*,
	\end{tikzcd}
	\end{equation*}
where $\gamma$ is surjective by the classical five-lemma. Let $\tau$ be the induced topology on $H^1(k,G^{\D})$ by $\gamma$. Thus, $(H^1(k,G^{\D}),\tau)$ becomes an abelian topological group and $\gamma$ becomes a topological quotient. In particular, it is continuous and strict. Since $H^1(k,G)$ is discrete by Proposition \ref{top alg grp}, the pairing \eqref{yoneda paring loc alg grp} is continuous (see Remark \ref{pairing discrete}). \\

Note now that $a_1$ is strict and continuous since $H^1(k,G)^*$ is Hausdorff and the induced map $\beta':\im a_1\to\im a_2$ is an isomorphism. Recalling that $\ker\delta$ is indiscrete, \cite[Corollary 3.19]{topart} implies that $b_1$ is continuous since $b_2$ is continuous. Thus, we obtain the following exact sequence of Hausdorff topological abelian groups
\begin{equation} \label{top loc alg seq}
	H^1(k,\pi_0(G)^D) \overset{(a_1)_{\Haus}}{\to} (H^1(k,G^{\D}),\tau)_{\Haus} \overset{(b_1)_{\Haus}}{\to} H^0(k,[L^D\to A^t])_{\Haus}.
\end{equation}
Indeed, since $\delta_{\Haus}$ is a topological isomorphism, $(b_1)_\Haus=\delta_{\Haus}^{-1}\circ b_2\circ\gamma_{\Haus}$ is continuous and strict.  Furthermore, the exact sequence \eqref{top loc alg seq} fits in the following commutative diagram of Hausdorff locally compact abelian topological groups
	\begin{equation*}
		\begin{tikzcd}
			H^{-1}(k,[L^D\to A^t])^\wedge \ar[r] \ar[d,"\vsim"'] & H^1(k,\pi_0(G)^D) \ar{r}{(a_1)_{\Haus}} \ar[d,"\beta","\vsim"'] & (H^1(k,G^{\D}),\tau)_{\Haus} \ar{r}{(b_1)_{\Haus}} \ar[d,"\vsim"',"\gamma_{\Haus}"] & H^0(k,[L^D\to A^t])_{\Haus} \ar[r] \ar[d,"\delta_{\Haus}","\vsim"'] & H^2(k,\pi_0(G)^D) \ar[d,"\vsim"'] \\
			H^{2}(k,G^0)^* \ar[r] & H^1(k,\pi_0(G))^* \ar{r}{a_2} & H^1(k,G)^* \ar{r}{b_2} & H^{1}(k,G^0)^* \ar[r] & H^0(k,\pi_0(G))^*,
		\end{tikzcd}
	\end{equation*}
whose rows are strict exact. Hence, $\gamma_\Haus$ is a topological isomorphism (see \cite[Proposition 2.8]{FulpGrif}), whence $(H^1(k,G^{\D}),\tau)_{\Haus}$ is profinite since $H^1(k,G)$ is discrete. Finally, since $H^1(k,\pi_0(G)^D)$ is finite and discrete, \cite[Proposition 3.22]{topart} implies that for every topology $\tau'$ such that 
\[ H^1(k,\pi_0(G)^D) \overset{a_1}{\to} (H^1(k,G^{\D}),\tau') \overset{b_1}{\to} H^0(k,[L^D\to A^t]), \]
is a strict exact sequence of abelian topological groups, there exists a topological isomorphism between $(H^1(k,G^{\D}),\tau')_{\Haus}$ and $(H^1(k,G^{\D}),\tau)_{\Haus}$, and these coincide with their corresponding profinite completions.
\end{proof}

\section{The truncated homology}

Let $k$ be a field of characteristic zero. Let $\phi:X\to\Spec k$ be a proper $k$-scheme. We define the \emph{$i$th group of truncated homology of $X$} as
\begin{equation} \label{ht}
	H_i(X,\Z)_\tau := R\Hom_{k_{\sm}}((\tau_{\leq 1}R\phi_*\G_{m,X})[i],\G_{m,k}) \cong H^i(k_{\sm},(\tau_{\leq 1}R\phi_*\G_{m,X})^\D).
\end{equation}
This definition can be regarded in terms of $\Ext$ groups, just as van Hamel's definition (see \cite[\S 2]{vH04}). We opt for this definition because it is more confortable to introduce the properties of truncated homology in terms of morphisms in $D^b(k_{\sm})$ and compositions of them.

In this section, we will define basic operations for the truncated homology giving the details of those properties that are not fully explained in van Hamel's paper (cf. \cite[\S 2.1]{vH04}). 

Throughout this section, $X$ and $Y$ will be proper $k$-schemes with structure maps $\phi:X\to\Spec k$ and $\psi:Y\to\Spec k$ respectively.


\subsection{The push-forward morphism}

Let $f:X\to Y$ be a morphism of $k$-schemes, that is a morphism of schemes $f:X\to Y$ such that $\phi=\psi\circ f$ or equivalently the following diagram commutes
\begin{equation*}
	\begin{tikzcd}[column sep=tiny]
		X \arrow[rr,"f"] \arrow{rd}[swap]{\phi} & & Y \arrow[dl,"\psi"] \\
		& \Spec(k) & .
	\end{tikzcd}
\end{equation*}
Firstly, we have the canonical isomorphism of derived functors
\begin{equation} \label{der comp}
	R\phi_* \overset{\sim}{\to} R\psi_*\circ Rf_*,
\end{equation}
associated to the derived functor of the composition (see \cite[p. 334, Proposition 13.3.13]{SK}), which induces by applying the truncation functor the following canonical isomorphism of functors
\begin{equation*}
	\tau_{\leq 1}R\phi_* \overset{\sim}{\to} \tau_{\leq 1}R\psi_* Rf_*.
\end{equation*}
Thus we obtain the following canonical isomorphism of groups
\begin{equation} \label{comp}
	H_i(X,\Z)_\tau \overset{\sim}{\to} \Hom_{D(k_{\sm})}((\tau_{\leq 1}R\psi_*(Rf_*\G_{m,X}))[i],\G_{m,k})
\end{equation}
for all $i\in\Z$. On the other hand, from the adjunction $f^*\dashv f_*$ and the canonical isomorphism $f^*\G_{m,Y}\cong\G_{m,X}$ (see \cite[p. 68, Remark II.3.1.(d)]{EC}), we obtain a canonical isomorphism 
\begin{equation} \label{adjunction}
\Hom_{S(X_{\sm})}(\G_{m,X},\G_{m,X}) \cong \Hom_{S(Y_{\sm})}(\G_{m,Y},f_*\G_{m,X}).
\end{equation}
Let us define $\alpha_f:\G_{m,Y}\to Rf_*\G_{m,X}$ as the composition
\begin{equation*} 
	\alpha_f:=(\G_{m,Y}\to f_*\G_{m,X}\to Rf_*\G_{m,X}),
\end{equation*}
where $\G_{m,Y}\to f_*\G_{m,X}$ is induced by $\id_X$ from \eqref{adjunction} and $f_*\G_{m,X}\to Rf_*\G_{m,X}$ is the canonical map $\tau_{\leq 0}Rf_*\G_{m,X}\to Rf_*\G_{m,X}$. Then we get a morphism
\begin{equation} \label{r alfa}
	\tau_{\leq 1}R\psi_*(\alpha_f): \tau_{\leq 1}R\psi_*\G_{m,Y} \to \tau_{\leq 1}R\psi_*(Rf_*\G_{m,X}).
\end{equation}
In this way, we get a canonical morphism
\begin{equation} \label{push-forward}
	f_*^{(i)}:H_i(X,\Z)_\tau\to H_i(Y,\Z)_\tau
\end{equation}
defined as the composition
\[ H_i(X,\Z)_\tau \overset{\sim}{\longrightarrow} \Hom_{D(k_{\sm})}((\tau_{\leq 1}R\psi_*(Rf_*\G_{m,X}))[i],\G_{m,k}) \overset{\theta}{\longrightarrow} H_i(Y,\Z)_\tau, \]
where the first isomorphism comes from \eqref{comp} and $\theta$ maps a morphism
\[ \eta: (\tau_{\leq 1}R\psi_*(Rf_*\G_{m,X}))[i] \to \G_{m,k} \]
to the morphism
\[ \eta\circ(\tau_{\leq 1}R\psi_*(\alpha_f)[i]):\tau_{\leq 1}R\psi_*\G_{m,Y}[i]\to\G_{m,k}, \]
where $\tau_{\leq 1}R\psi_*(\alpha_f)$ is the map defined in \eqref{r alfa}. The morphism \eqref{push-forward} is called the \emph{push-forward of $f$ in the $i$th group of truncated homology}. It is direct from the definition that this construction is functorial, that is, $(g\circ f)^{(i)}_*$ and $g_*^{(i)}\circ f_*^{(i)}$ are canonically isomorphic for every pair of composable maps $g$ and $f$ of proper $k$-schemes (this isomorphism comes from \eqref{der comp}).


\subsection{The truncated homology of a point}

For any morphism $\pi:\Spec L\to\Spec k$ induced by a finite separable extension $L/k$, we have that $\tau_{\leq 1}R\pi_*\G_{m,L}$ is quasi-isomorphic to $\pi_*\G_{m,L}$. Indeed, the sheaf $R^1\pi_*\G_{m,L}$ agrees with the étale sheafification of the presheaf
\[ U \mapsto H^1((U_L)_{\sm},\G_{m,U_L}) \] 
(see \cite[p. 253, Theorem 9.2.5]{FGA}). Whence we conclude by Hilbert's theorem 90. 
Thus, from the exact triangle
\[ \pi_*\G_{m,L} \to \tau_{\leq 1}R\pi_*\G_{m,L} \to R^1\pi_*\G_{m,L} \to \pi_*\G_{m,L}[1], \]
we conclude that the canonical map $\pi_*\G_{m,L} \to \tau_{\leq 1}R\pi_*\G_{m,L}$ is an isomorphism in $D(k_{\sm})$. Then we have that
\[ H_i(\Spec L,\Z)_\tau \cong H^i(k_{\sm},R\sHom_{D(k_{\sm})}(\pi_*\G_{m,L},\G_{m,k})). \]
Now, observe that $\pi_*\G_{m,L}$ is represented by the Weil restriction $R_{L/k}(\G_{m,L})$, which is also a torus since $L/k$ is separable. Thus, by Proposition \ref{dual suave}, we have that
\[ H_i(\Spec L,\Z)_\tau \cong H^i(k_{\sm},R_{L/k}(\G_{m,L})^D). \]
Let $M/k$ be the Galois closure of $L/k$. Then, if $\Gamma:=\Gal(M/k)$ and $\Delta:=\Gal(M/L)$, we have that $R_{L/k}(\G_{m,L})^D$ corresponds to the $\Gamma$-module $\Z[\Gamma/\Delta]$. Then, we have that
\[ H_i(\Spec L,\Z)_\tau \cong H^i(k,\Z[\Gamma/\Delta]), \]
where the right group is the group of Galois cohomology of the $\Gal(k)$-module $\Z[\Gamma/\Delta]$. Hence, by Shapiro's lemma, we conclude that
\begin{equation}\label{homology of a point}
	H_i(\Spec L,\Z)_\tau \cong H^i(L,\Z),
\end{equation}
where the last group is a group of Galois cohomology of the trivial $\Gal(L)$-module $\Z$. \\

Now, by the definition of the push-forward \eqref{push-forward}, we have that the map
\[ \pi_*^{(i)} : H_i(\Spec L,\Z)_\tau \to H_i(\Spec k,\Z)_\tau \]
is defined as the composition
\[ H_i(\Spec L,\Z)_\tau \overset{\sim}{\longrightarrow} \Hom_{D(k_{\sm})}(R_{L/k}(\G_{m,X})[i],\G_{m,k}) \overset{\circ\alpha_\pi[i]}{\longrightarrow} H_i(\Spec k,\Z)_\tau, \]
where $\alpha_\pi:\G_{m,k}\to R_{L/k}(\G_{m,k})$ is the diagonal embedding. Hence, the push-forward $\pi_*^{(i)}$ agrees with the corestriction map on Galois cohomology 
\[ \Cores^{(i)}:H^i(L,\Z) \to H^i(k,\Z), \]
which in degree zero 
\[ \pi_*^{(0)} : H_0(\Spec L,\Z)_\tau \to H_0(\Spec k,\Z)_\tau \]
is just multiplication by $[L:k]$ in $\Z$ via the identification $H_0(\Spec L,\Z)_\tau=H^0(L,\Z)=\Z$.


\subsection{The degree map}

For every proper $k$-scheme $\phi:X\to \Spec k$, the pushforward $\phi_*^{(0)}$ \eqref{push-forward} defines a canonical map
\begin{equation*} 
	\deg_{\phi}:H_0(X,\Z)_\tau \to \Z
\end{equation*}
via the identification $H_0(X,\Z)_\tau\cong \Z$ given in \eqref{homology of a point}. We will call $\deg_{\phi}$ the \emph{degree map for $\phi$}. Sometimes we will just write $\deg_X$ instead of $\deg_\phi$ and we will call it the \emph{degree map for $X$}.

\begin{remark} \label{deg point}
	Note that $\deg$ is surjective whenever $X(k)\neq\varnothing$. Indeed, a morphism $\Spec k \to X$ defines a section for the degree map.
\end{remark}

\begin{definition} \label{pseudo-index}
	For a proper $k$-scheme $\phi:X\to \Spec k$, we define \emph{the truncated pseudo-index} $\pseudo(X)$ of $X$ as the order of $\coker\deg_\phi$.
\end{definition}

Observe that the pushforward $\phi_*^{(i)}$ \eqref{push-forward} comes from the natural map $\G_{m,k}\to\tau_{\leq 1}R\phi_*\G_{m,k}$ once one takes cohomology on the dual morphism $(\tau_{\leq 1}R\phi_*\G_{m,k})^\D \to \Z$, that is,
\[ \deg_\phi: H^0(k_{\sm},(\tau_{\leq 1}R\phi_*\G_{m,k})^\D) \to \Z. \]
Thus, we obtain the following description of $\ker\deg_\phi$.

\begin{proposition} \label{ker coker deg}
	Let $X$ be a proper $k$-variety. We have the following exact sequence of abelian groups 
	\[ 0 \to \Ext^1_{k_{\sm}}(\Pic_{X/k},\G_{m,k}) \to H_0(X,\Z)_\tau \overset{\deg_\phi}{\to} \pseudo(X)\Z \to 0. \]
\end{proposition}

\begin{proof}
	We have the following exact triangle
	\[ \G_{m,k}\to\tau_{\leq 1}R\phi_*\G_{m,k}\to\Pic_{X/k}[-1]\to\G_{m,k}[1]. \]
	Applying the derived Cartier dual $(-)^\D$ to the triangle above, we obtain the following exact triangle in $D(k_\sm)$
	\[ \Pic_{X/k}^\D[1] \to (\tau_{\leq 1}R\phi_*\G_{m,k})^\D \to \Z \to \Pic_{X/k}^\D. \]
	Applying $R\Gamma(k_\sm,-)$ to the triangle here above, we obtain the following exact sequence of abelian groups
	\[ 0 \to H^0(k_\sm,\Pic_{X/k}^\D[1]) \to H^0(k_\sm,(\tau_{\leq 1}R\phi_*\G_{m,k})^\D) \overset{\deg_\phi}{\to} \Z. \]
	Now, $H^0(k_\sm,\Pic_{X/k}^\D[1])$ is isomorphic to $H^1(k_\sm,\Pic_{X/k}^\D)$ and, this last group is canonically isomorphic to $\Ext^1_{k_{\sm}}(\Pic_{X/k},\G_{m,k})$.
\end{proof}

Another formal property of the pseudo-index is the following.

\begin{proposition} \label{divisibility pseudo index}
	Let $f:X\to Y$ be a morphism of proper $k$-varieties. Then, $\pseudo(Y)$ divides $\pseudo(X)$.	
\end{proposition}

\begin{proof}
	That is a formal consequence of the functoriality of the pushforward on the truncated homology. We have the following commutative triangle
	\[ \begin{tikzcd}
		H_0(X,\Z)_\tau \ar{rd}[swap]{\deg_X} \ar{rr}{f_*^{(0)}} && H_0(X,\Z)_\tau \ar{ld}{\deg_Y} \\ & \Z. &
	\end{tikzcd} \]
	The $\ker$-$\coker$ exact sequence (see \cite[1.2 Product-Lemma]{kercoker}) yields a surjection
	\[ \coker\deg_X \twoheadrightarrow \coker\deg_Y, \]
	whence, $\pseudo(Y)$ divides $\pseudo(X)$.
\end{proof}

Let $\phi:X\to \Spec k$ be a proper $k$-scheme and $\pi:\Spec L\to \Spec k$ the morphism induced by a finite extension $L/k$ of $k$. By the functoriality of the push-forward in truncated homology, the following cartesian diagram
\[ \begin{tikzcd}
	X_L \arrow{r}{\pi_L} \arrow{d}[swap]{\phi_L} & X \arrow{d}{\phi} \\
	\Spec L \arrow{r}[swap]{\pi} & \Spec k
\end{tikzcd} \]
yields the following commutative diagram of abelian groups
\[ \begin{tikzcd}
	H_0(X_L,\Z)_\tau \arrow{r}{{\pi_L}_*^{(0)}} \arrow{d}[swap]{\deg_{\phi_L}} & H_0(X,\Z)_\tau \arrow{d}{\deg_{\phi}} \\
	\Z \arrow{r}[swap]{\cdot[L:k]} & \Z .
\end{tikzcd} \]
Thus we obtain the following commutative diagram of abelian groups with exact rows
\begin{equation} \label{deg diagram}
	\begin{tikzcd}
		& H_0(X_L,\Z)_\tau \arrow{r}{{\pi_L}_*^{(0)}} \arrow{d}{\deg_{\phi_L}} & H_0(X,\Z)_\tau \arrow{d}{\deg_{\phi}} \ar{r} & \coker({\pi_L}_*^{(0)}) \arrow{r} \arrow{d}{\delta_L} & 0 \\
		0 \ar{r} & \Z \arrow{r}[swap]{\cdot[L:k]} & \Z \ar{r} & \frac{\Z}{[L:k]\Z} \ar{r} & 0.
	\end{tikzcd}
\end{equation}
Applying the Snake Lemma to \eqref{deg diagram}, we obtain the following exact sequence of abelian groups
\begin{equation} \label{deg sequence}
	0 \to \ker({\pi_L}_*^{(0)}) \to \ker \deg_{\phi_L} \to \ker \deg_{\phi} \to \ker \delta_L \to \coker \deg_{\phi_L} \to \coker \deg_{\phi} \overset{q}{\to} \coker \delta_L \to 0.
\end{equation}
In this way, we obtain the following proposition.

\begin{proposition} \label{deg surjective}
	Let $\phi:X\to\Spec(k)$ be a proper $k$-scheme. Then $\pseudo(X)$ divides the index $I(X)$ of $X$. In particular, if $X$ has a zero-cycle of degree 1, then $\deg_\phi$ is surjective.
\end{proposition}

\begin{proof}
	Suppose that $X$ has an $L$-rational point for some finite extension $L/k$. From the exact sequence \eqref{deg sequence} and Remark \ref{deg point}, $q:\coker \deg_{\phi} \to \coker \delta_L$ is an isomorphism. Hence $\coker \deg_{\phi}$ is a quotient of the cyclic group $\frac{\Z}{[L:k]\Z}$ and, therefore, its order $\pseudo(X)$ divides $[L:k]$. Then, $\pseudo(X)$ divides the degree of all zero-cycles of $X$, whence $\pseudo(X)$ divides the index $I(X)$ of $X$.
\end{proof}


\subsection{Yoneda pairing} \label{yp}

For every $k$-scheme $X$, we have the Yoneda pairing 
\begin{align} \label{yoneda}
	\begin{split}
		H_i(X,\Z)_\tau \times H^j(k_{\sm},\tau_{\leq 1}R\phi_*\G_{m,X}) &\to H^{j-i}(k_{\sm},\G_{m,k}) \\
		(\gamma,\omega) &\mapsto \gamma\cdot\omega.
	\end{split}
\end{align}
induced for arbitrary $i,j$ via the canonical map
\begin{align*}
	H_i(X,\Z)_\tau=\Hom_{D(k_{\sm})}(\tau_{\leq 1}R\phi_*\G_{m,X},\G_{m,k}[-i]) &\to \Hom(H^{j}(k_{\sm},\tau_{\leq 1}R\phi_*\G_{m,X}),H^{j-i}(k_{\sm},\G_{m,k})),
\end{align*}
which maps a morphism $\gamma$ to the cohomology map $H^j(\gamma)$. Here we canonically identify $H^j(k_{\sm},\G_{m,k}[-i])$ and $H^{j-i}(k_{\sm},\G_{m,k})$. From the definition of the Yoneda pairing \eqref{yoneda}, we see that it is functorial on $X$. Thus, for every morphism of $k$-schemes $f:X\to Y$, the following formula holds
\[ f_*\gamma\cdot\omega=\gamma\cdot f^*\omega, \]
which is deduced from the definition of the push-forward in truncated homology and pull-back in sheaf cohomology. Moreover, the pairing \eqref{yoneda} is functorial on $X$ since it is defined as composition of morphisms in $D(k_{\sm})$.


\section{Computations of truncated homology}

In this section, $k$ will still be a field of characteristic zero and $\phi:X\to\Spec k$ will be a proper $k$-variety. Recall that the Picard variety $\Pic_{X/k}^0$ of $X$ is an extension
\begin{equation} \label{pic0} 
	0 \to L \to \Pic_{X/k}^0 \to A \to 0
\end{equation}
of an abelian variety $A$ by a connected linear group $L$.


\subsection{A filtration on the truncated derived direct image of $\G_{m,X}$} \label{fil}

The goal of this section is to compute the truncated homology of $X$. In order to do that, we will define a filtration of 
\[ \s{C}_X :=\tau_{\leq 1}R\phi_*\G_{m,X}. \]
Recall that a filtration on derived categories is just a sequence of morphisms since in this context we do not have a notion of subobject. We will construct this filtration taking into account that $R^1\phi_*\G_{m,X}$ is represented by the locally algebraic group $\Pic_{X/k}$, which is an extension of the finitely generated group $\NS_{X/k}$ by the algebraic group $\Pic_{X/k}^0$. Firstly, we define
\[ \s{F}_2:=\s{C}_X, \quad \s{F}_0:=\tau_{\leq 0}\s{C}_X \quad \text{and} \quad \s{F}_{-1}=0. \]
We note that
\[ \s{F}_0 \cong H^0(\s{C}_X) \cong \phi_*\G_{m,X} \cong \G_{m,k}, \]
where the last isomorphism comes from the fact that $X$ is proper and geometrically integral (see \cite[\href{https://stacks.math.columbia.edu/tag/0BUG}{Lemma 0BUG}]{stacks-project}). Now, there is a canonical map
\[ \s{F}_2\to\NS_{X/k}[-1], \]
which is defined as the composition
\[ \s{F}_2 \to \tau_{\geq 1}\s{F}_2 \to \NS_{X/k}[-1]. \]
where the last arrow is induced by the canonical projection $\Pic_{X/k}\to\NS_{X/k}$ using the fact that 
\[ \tau_{\geq 1}\s{F}_2 = \tau_{\geq 1}\s{C}_X \cong H^1(\s{C}_X)[-1] = (R^1\phi_*\G_{m,X})[-1] \cong \Pic_{X/k}[-1]. \]
Thus, we define
\[ \s{F}_1:=\cone(\s{F}_2\to\NS_{X/k}[-1])[-1], \]
and, thereby, it yields the following exact triangle
\begin{equation*} 
	\s{F}_1 \to \s{F}_2 \to \NS_{X/k}[-1] \to \s{F}_1[1],
\end{equation*}
which also defines the morphism $\s{F}_1 \to \s{F}_2$ in the filtration. The morhpism $\s{F}_0\to\s{F}_1$ is defined as follows: we have the following diagram
\begin{equation} \label{g1}
	\begin{tikzcd}
		\s{F}_0 \ar[r] \ar[d] & 0 \ar[r] \ar[d] & \s{F}_0[1] \ar[r,equal] \ar[d,dashrightarrow] & \s{F}_0[1] \ar[d] \\
		\s{F}_2 \ar[r] & \NS_{X/k}[-1] \ar[r] & \s{F}_1[1] \ar[r] & \s{F}_2[1],
	\end{tikzcd}
\end{equation}
where the first vertical map is the canonical map $\tau_{\leq 0}\s{C}_X \to \s{C}_X$ and the second one is the trivial map. The diagram \eqref{g1} induces a morphism of exact triangles, that is, the dashed arrow can be completed in the diagram. Since $\Hom_{D(k_{\sm})}(0,\s{F}_2)=0$ and $\Hom_{D(k_{\sm})}(\s{F}_0[1],\NS_{X/k}[-1])=0$ (see \cite[p. 324, Proposition 13.1.16]{SK}), this map is in fact unique (see \cite[p. 246, Proposition 10.1.17]{SK}). Thus, shifting by $-1$ the dashed arrow in diagram \eqref{g1}, we define the morphism $\s{F}_0 \to \s{F}_1$ in the filtration. \\

Now, for every $i\in\{0,1,2\}$ we define the \emph{$i$th graded piece $\s{G}_i$} as cone of the map $\s{F}_{i-1}\to\s{F}_i$. Since $\s{F}_{-1}$ is trivial by definition, the graded piece $\s{G}_0$ agrees with $\s{F}_0$. This yields the following two exact triangles on $D^b(k_{\sm})$
\begin{align} 
	&\s{F}_{0}\to\s{F}_1\to\s{G}_1\to\s{F}_{0}[1] \label{triangle1} \\
	&\s{F}_{1}\to\s{F}_2\to\s{G}_2\to\s{F}_{1}[1]. \label{triangle2}
\end{align}
Thus, we obtain a 3-term filtration of $\s{C}_X$
\[ 0 = \s{F}_{-1}\to\s{F}_0\to\s{F}_1\to\s{F}_2=\s{C}_X, \]
whose graded pieces are the following
\begin{equation} \label{graded pieces}
	\s{G}_i\cong\begin{cases} \G_{m,k} & \text{if }i=0, \\ \Pic_{X/k}^0[-1] & \text{if }i=1, \\ \NS_{X/k}[-1] & \text{if }i=2. \end{cases}
\end{equation}

\subsection{The dual cofiltration}

Recall that $H_i(X,\Z)_\tau=H^i(k_{\sm},\s{C}_X^{\D})$ (see \eqref{ht}). Let us apply the derived Cartier dual $(-)^\D$ \eqref{derived Cartier dual} to the filtration $\s{F}_\bullet$ of $\s{C}_X$ defined in the previous section. We obtain the following cofiltration of $\s{C}_X^{\D}$
\[ \s{C}_X^{\D}=\s{F}_2^{\D}\to\s{F}_1^{\D}\to\s{F}_0^{\D}\to\s{F}_{-1}^{\D}=0. \]
Moreover, applying $(-)^\D$ to triangles \eqref{triangle1} and \eqref{triangle2}, we obtain the following two exact triangles
\begin{align}
	&\s{G}_1^{\D}\to\s{F}_1^{\D}\to\s{F}_{0}^{\D}\to\s{G}_1^{\D}[1] \label{dual triangle1} \\
	&\s{G}_2^{\D}\to\s{F}_2^{\D}\to\s{F}_{1}^{\D}\to\s{G}_2^{\D}[1] \label{dual triangle2}.
\end{align}
Since each graded piece $\s{G}_i$ is quasi-isomorphic to a complex concentrated in a single degree, we can make a direct computation of $\s{G}_i^{\D}$. Since $\G_{m,k}$ is a torus and $\NS_{X/k}$ is a finitely generated group scheme, we have that
\[ \s{G}_0^{\D}\cong (\G_{m,k})^D\cong\Z \quad \text{and} \quad  \s{G}_2^{\D}\cong (\NS_{X/k}[-1])^{\D} \cong (\NS_{X/k})^D[1]. \]
Now, the graded piece $\s{G}_1$ is quasi-isomorphic to $\Pic_{X/k}^0[-1]$. Then, applying $(-)^{\D}$ to \eqref{pic0}, we obtain that
\[ \s{G}_1^{\D} \cong (\Pic_{X/k}^0[-1])^{\D} \cong (\Pic_{X/k}^0)^{\D}[1] \cong [L^D \to A^t], \]
where the last complex is concentrated in degrees $-1$ and 0 (see Proposition \ref{dual suave}.d). In summary, the Cartier dual of the graded pieces $\s{G}_i$ are the following
\begin{equation} \label{dual graded pieces}
	\s{G}_i^{\D}\cong\begin{cases} \Z & \text{if }i=0, \\ [L^D\to A^t] & \text{if }i=1, \\ (\NS_{X/k})^D[1] & \text{if }i=2. \end{cases}
\end{equation}
Note that, unlike van Hamel's situation (see \cite[\S 2.2]{vH04}), $\s{G}_1^{\D}$ is not located at a single degree.

\subsection{A formal description of the zeroth group of truncated homology}

By composition, the sequence $\s{F}_0\to\s{F}_1\to\s{F}_2$ yields the following commutative diagram
\begin{equation} \label{cdfp}
	\begin{tikzcd}
		\s{F}_0 \ar{r} \ar[d,equal] & \s{F}_1 \ar{r} \ar{d} & \s{G}_1 \ar{r} \ar{d} & \s{F}_0[1] \ar[d,equal] \\
		\s{F}_0 \ar{r} & \s{F}_2 \ar{r} \ar{d} & \Pic_{X/k}[-1] \ar{r} \ar{d} & \s{F}_0[1] \\
		& \s{G}_2 \ar[r,equal] \ar{d} & \NS_{X/k}[-1] \ar{d} & \\
		& \s{F}_1[1] \ar{r} & \s{G}_1[1] & 
	\end{tikzcd}
\end{equation}
where all rows and columns are exact triangles. The arrow $\s{G}_1\to\Pic_{X/k}[-1]$ is the map induced by the canonical embedding $\Pic_{X/k}^0\to\Pic_{X/k}$. Then, taking Cartier dual to \eqref{cdfp}, we obtain the commutative diagram
\begin{equation} \label{dual cdfp}
	\begin{tikzcd}
		\NS_{X/k}^{\D}[1] \ar{d} \ar[r,equal] & \s{G}_2^{\D} \ar{d} & & \\ 
		\Pic_{X/k}^{\D}[1] \ar{r} \ar{d} & \s{F}_2^{\D} \ar{r} \ar{d} & \s{F}_0^{\D} \ar{r} \ar[d,equal] & \Pic_{X/k}^{\D}[2] \ar{d} \\
		\s{G}_1^{\D} \ar{r} \ar{d} & \s{F}_1^{\D} \ar{r} \ar{d} & \s{F}_0^{\D} \ar{r} & \s{G}_1^{\D}[1] \\
		\NS_{X/k}^{\D}[2] \ar[r,equal] & \s{G}_2^{\D}[1] & &
	\end{tikzcd}
\end{equation}
with all rows and columns being exact triangles. Applying the derived functor $R\Gamma(k_{\sm},-)$ to \eqref{dual cdfp}, we obtain the following commutative diagram with exact rows and columns
\begin{equation*}
	\begin{tikzcd}[column sep=small]
		& H^1(k_{\sm},\NS_{X/k}^{\D}) \ar{r}{\sim} \ar{d} & H^0(k_{\sm},\s{G}_2^{\D}) \ar{d} \\
		0 \ar{r} & H^1(k_{\sm},\Pic_{X/k}^{\D}) \ar{r} \ar{d} & H_0(X,\Z)_{\tau} \ar{r} \ar{d} & H^0(k_{\sm},\s{F}_0^{\D}) \ar{r} \ar[d,equal] & H^2(k_{\sm},\Pic_{X/k}^{\D}) \ar{d} \ar{r} & \cdots \\
		0 \ar{r} & H^0(k_{\sm},\s{G}_1^{\D}) \ar{r} \ar{d} & H^0(k_{\sm},\s{F}_1^{\D}) \ar{r} \ar{d} & H^0(k_{\sm},\s{F}_0^{\D}) \ar{r} & H^1(k_{\sm},\s{G}_1^{\D}) \ar{r} & \cdots \\
		& H^2(k_{\sm},\NS_{X/k}^{\D}) \ar{r}{\sim} & H^1(k_{\sm},\s{G}_2^{\D})
	\end{tikzcd}
\end{equation*}
where the left zeros come from $H^{-1}(k_{\sm},\s{F}_0^{\D})$ since $\s{F}_0^{\D}\cong \G_{m,k}^{D}\cong \Z$. Then, $\s{G}_1^{\D}$ is quasi-isomorphic to the generalized 1-motive $[L^D\to A^t]$, where $L$ and $A$ are defined in \eqref{pic0}. Thus we have the following commutative diagram with exact rows and columns
\begin{equation} \label{deg map 3}
	\begin{tikzcd}[column sep=small]
		& H^1(k_{\sm},\NS_{X/k}^D) \ar[r,equal] \ar{d} & H^1(k_{\sm},\NS_{X/k}^D) \ar{d} \\
		0 \ar{r} & H^1(k_{\sm},\Pic_{X/k}^{\D}) \ar{r} \ar{d} & H_0(X,\Z)_{\tau} \ar{r}{\deg} \ar{d} & \Z \ar{r} \ar[d,equal] & H^2(k_{\sm},\Pic_{X/k}^{\D}) \ar{d} \ar{r} & \cdots \\
		0 \ar{r} & H^0(k_{\sm},[L^D\to A^t]) \ar{r} \ar{d} & H^0(k_{\sm},\s{F}_1^{\D}) \ar{r} \ar{d} & \Z \ar{r} & H^1(k_{\sm},[L^D\to A^t]) \ar{r} & \cdots \\
		& H^2(k_{\sm},\NS_{X/k}^D) \ar[r,equal] & H^2(k_{\sm},\NS_{X/k}^D).
	\end{tikzcd}
\end{equation}
Using the canonical identification $H^i(k_{\sm},\s{F}^{\D})\cong\Ext_{k_{\sm}}^i(\s{F},\G_{m,k})$ for any complex of sheaves $\s{F}$, we have that $H^i(k_{\sm},\Pic_{X/k}^{\D})\cong\Ext^i_{k_{\sm}}(\Pic_{X/k},\G_{m,k})$. Then the left hand column in the diagram \eqref{deg map 3}
\[ H^1(k_{\sm},\NS_{X/k}^D) \to H^1(k_{\sm},\Pic_{X/k}^{\D}) \to H^0(k_{\sm},[L^D\to A^t]) \to H^2(k_{\sm},\NS_{X/k}^D) \]
can be regarded as the exact sequence of $\Ext$ groups
\[ \Ext_{k_{\sm}}^1(\NS_{X/k},\G_{m,k}) \to \Ext_{k_{\sm}}^1(\Pic_{X/k},\G_{m,k}) \to \Ext_{k_{\sm}}^1(\Pic^0_{X/k},\G_{m,k}) \to \Ext_{k_{\sm}}^2(\NS_{X/k},\G_{m,k}) \]
obtained from the exact sequence
\[ 0 \to \Pic_{X/k}^0 \to \Pic_{X/k} \to \NS_{X/k} \to 0. \]


\section{Generalised Tate duality} \label{section Tate}

In this chapter, we will prove that the filtration pieces $\s{F}_i$ satisfy Tate duality when $k$ is $p$-adic and $\phi:X\to\Spec(k)$ is any proper $k$-variety. We will strongly use that for the graded pieces $\s{G}_i$ there exist Tate duality results over $p$-adic fields (see \cite[Corollary 2.3]{ADT} and \cite[Theorem 0.1]{HS05}).\\

Since $k$ has characteristic zero, the algebraic group $\Pic_{X/k}^0$ is an extension of an abelian variety $A$ by a linear algebraic group $L$, which is also decomposed as the product of a torus $T$ and finitely many copies of the additive group $\G_a$. The following result tells us that, for cohomological purposes, we may work thinking that $\Pic_{X/k}^0$ is a semiabelian variety.

\begin{proposition} \label{ga}
	Let $X/k$ be a proper variety over a field $k$ of characteristic zero. Let $U$ be the unipotent part of $\Pic_{X/k}^0$. The canonical projection $\pi:\Pic_{X/k}^0\to\Pic_{X/k}^0/U$ induces the following isomorphisms of abelian groups:
	\begin{enumerate}[nosep, label=\alph*)]
	\item $H^i(k,\Pic_{X/k}^0) \cong H^i(k,\Pic_{X/k}^0/U)$ for all $i\neq0$. 
	\item $H^i(k,(\Pic_{X/k}^0)^\D) \cong H^i(k,(\Pic_{X/k}^0/U)^\D)$ for all $i$.
	\end{enumerate}
	Moreover, when $k$ is $p$-adic, the isomorphisms above are isomorphisms of topological groups. Furthermore, in this case, $H^0(k,\Pic_{X/k}^0)^\wedge$ is topologically isomorphic to $H^0(k,\Pic_{X/k}^0/U)^\wedge$.
\end{proposition}

\begin{proof}
	Since $k$ is of characteristic zero, $U$ is isomorphic to $\G_{a,k}^n$ for some $n\in\Z$ non-negative. Then, $a)$ follows immediately from the fact that $H^i(k,\G_a)=0$ for all $i\neq0$; and $b)$ follows from the fact that $H^i(k,\G_a^D)=0$ for all $i$ (see \cite[Proposition 2.5.1]{Rosengarten}). When $k$ is $p$-adic, every cohomology group in $a)$ and every cohomology group in $b)$, except for $i=1$, is discrete by Proposition \ref{top motives}. Whence these isomorphisms are topological isomorphisms. Recall that $(\Pic_{X/k}^0)^\D$ is the shift by $-1$ of a generalized 1-motive of the form $[L^D\to A^t]$ (see Proposition \ref{dual suave}). Hence, the proof follows from Remark \ref{h0 motives}.b.
\end{proof}

In order to prove the duality results for the filtration pieces, we will strongly use inductive arguments on the following exact sequences
\begin{equation} \label{pairing}
	\begin{tikzcd}[column sep=1.5em]
		\cdots \ar{r} & H^{r-1}(k,\s{G}_i) \ar{r} & H^{r}(k,\s{F}_{i-1}) \ar{r} & H^{r}(k,\s{F}_i) \ar{r} & H^{r}(k,\s{G}_i) \ar{r} & \cdots \\
		\cdots & H^{r+1}(k,\s{G}_i^{\D}) \ar[l] & H^{r}(k,\s{F}_{i-1}^{\D}) \ar[l] & H^{r}(k,\s{F}_i^{\D}) \ar[l] & H^{r}(k,\s{G}_i^{\D}) \ar[l] & \cdots, \ar[l]
	\end{tikzcd}
\end{equation}
for $i=1$ or 2, induced from the exact triangles \eqref{triangle1}, \eqref{triangle2}, \eqref{dual triangle1} and \eqref{dual triangle2}.


\subsection{Tate duality for the graded pieces}

In the following table, we recall the graded pieces associated to the filtration $\s{F}_\bullet$ defined in  \S\ref{fil} and their Cartier duals.
\begin{equation*}
	\begin{tabular}{|c|c|c|}
		\hline
		$i$	& $\s{G}_i$		& $\s{G}_i^{\D}$ \\
		\hline
		0	& $\G_{m,k}$		& $\Z$ \\
		1	& $\Pic^0_{X/k}[-1]$	& $[L^D\to A^t]$ \\
		2	& $\NS_{X/k}[-1]$	& $\NS_{X/k}^D[1]$ \\
		\hline
	\end{tabular}
\end{equation*}
Here, $\s{G}_1^{\D}$ is located at degrees $-1$ and $0$. The cohomology groups of the graded pieces are non-trivial in finitely cases as we recall in the following proposition.

\begin{proposition} \label{vanishing graded}
	Let $r\in\Z$.
	\begin{enumerate}[nosep, label=\alph*)]
		\item $H^r(k,\s{G}_0^{\D}) = H^{2-r}(k,\s{G}_0) = \{0\}$ for $r\notin\{0,2\}$.
		\item $H^r(k,\s{G}_i^{\D}) = H^{2-r}(k,\s{G}_i) = \{0\}$ for $i=1$ or 2 and $r\notin\{-1,0,1\}$.
	\end{enumerate}
\end{proposition}

\begin{proof}
	For see \cite[Theorem I.2.1]{ADT} and \cite[Lemma 2.1 and Theorem 2.3]{HS05}.
\end{proof}

Moreover, the graded pieces satisfy a Tate duality result as follows.

\begin{proposition} \label{tate}
	Let $r\in\{-1,0,1,2\}$ and $i=0,1$ or 2. Then the Yoneda pairing
	\[ H^r(k,\s{G}_i^{\D}) \times H^{2-r}(k,\s{G}_i) \to H^2(k,\G_{m,k})\cong\Q/\Z \]
	is non-degenerate for every $i\in\{0,1,2\}$ and $r\in\Z$. The above pairing induces the following continuous and perfect pairings between profinite abelian groups and discrete abelian torsion groups respectively:
	\begin{enumerate}[nosep, label=\alph*)]
		\item For $i=0,1$ or 2,
		\[ H^0(k,\s{G}_i^{\D})^\wedge \times H^{2}(k,\s{G}_i) \to \Q/\Z \]
		and a similar pairing with $\s{G}_i$ and $\s{G}_i^{\D}$ interchanged. When $i=2$, both cohomology groups are finite.
		\item For $i=1$ or 2,
		\[ H^1(k,\s{G}_i)^{\wedge} \times H^1(k,\s{G}_i^{\D}) \to \Q/\Z \]
		and
		\[ H^{-1}(k,\s{G}_i^{\D})^\wedge \times H^{3}(k,\s{G}_i) \to \Q/\Z. \]
	\end{enumerate}
\end{proposition}

\begin{proof}
	For $i\in\{0,2\}$, the result follows from Hilbert's Theorem 90 and Tate duality for finitely generated groups (see \cite[Theorem I.2.1]{ADT}). For $i=1$, the proposition follows from Proposition \ref{ga}, Tate duality for Deligne 1-motives (see \cite[Lemma 2.1 and Theorem 2.3]{HS05}) and Proposition \ref{h-1}.
\end{proof}

Now, recall the triangles \eqref{triangle1} and \eqref{triangle2}
\begin{align*} 
	&\s{F}_{0}\overset{\alpha_1}{\to} \s{F}_1\to\s{G}_1\to\s{F}_{0}[1]  \\
	&\s{F}_{1}\overset{\beta}{\to}\s{F}_2\to\s{G}_2\to\s{F}_{1}[1].
\end{align*}
Composing the maps $\alpha_1$ and $\beta$, we obtain a map
\[ \alpha_2 := \s{F}_{0}\overset{\alpha_1}{\to} \s{F}_1 \overset{\beta}{\to}\s{F}_2 \]
from $\s{F}_0$ to $\s{F}_2$. The maps
\begin{equation} \label{alfas}
	\alpha_1:\s{F}_0 \to \s{F}_1 \quad \text{and} \quad \alpha_2:\s{F}_0 \to \s{F}_2
\end{equation}
induces isomorphisms between some cohomology groups of the filtration pieces.

\begin{corollary} \label{tatecoro}
	Let $i=1$ or 2. 
	\begin{enumerate}[nosep, label=\alph*)]
		\item The maps $\alpha_i:\s{F}_0\to\s{F}_i$ \eqref{alfas} and $\alpha_i^\D:\s{F}_i^\D\to\s{F}_0^\D$ induce the following isomorphisms
		\[ H^0(k,\s{F}_0) \overset{\sim}{\to} H^0(k,\s{F}_i) \quad \text{and} \quad H^2(k,\s{F}_i^{\D}) \overset{\sim}{\to} H^2(k,\s{F}_0^\D). \]
		\item The map $\s{F}_1\to\s{G}_1$ and its dual map $\s{G}_1^\D\to\s{F}_1^\D$ induce the following isomorphisms
		\[ H^3(k,\s{G}_1) \overset{\sim}{\to} H^3(k,\s{F}_1) \quad \text{and} \quad H^{-1}(k,\s{F}_1^{\D}) \overset{\sim}{\to} H^{-1}(k,\s{G}_1^{\D}). \]
	\end{enumerate}
\end{corollary}

\begin{proof}
	Point $a)$ follows directly from the exact sequences \eqref{pairing} and the fact that $H^0(k,\s{G}_1)$ and $H^0(k,\s{G}_2)$ are trivial. For $b)$, since $k$ is a local field, it has strict cohomological dimension 2. Hence, the groups 
	\[ H^{2-r}(k,\s{F}_0)=H^{2-r}(k,\G_{m,k}), \quad H^r(k,\s{F}_0^{\D})=H^r(k,\Z) \] 
	are trivial for $r\notin\{0,1,2\}$. Thus, by the exactness of the sequences \eqref{pairing}, we have the isomorphisms
	\[ H^{3}(k,\s{F}_1)\cong H^3(k,\s{G}_1) \quad \text{and} \quad H^{-1}(k,\s{F}_1^{\D})\cong H^{-1}(k,\s{G}_1^{\D}). \]
\end{proof}

\begin{corollary} \label{vanish}
	Let $r\in\Z$.
	\begin{enumerate}[nosep, label=\alph*)]
		\item $H^{2-r}(k,\s{F}_0) = H^r(k,\s{F}_0^{\D}) = \{0\}$ for $r\notin\{0,2\}$.
		\item $H^{2-r}(k,\s{F}_i) = H^r(k,\s{F}_i^{\D}) = \{0\}$ for $i=1$ or 2 and $r\notin\{-1,0,1,2\}$.
	\end{enumerate}
\end{corollary}

\begin{proof}
	For $a)$, we note that $\G_{m,k}=\s{F}_0=\s{G}_0$. For $b)$, we use the exact sequences
	\[ H^{2-r}(k,\s{F}_{i-1}) \to H^{2-r}(k,\s{F}_i) \to H^{2-r}(k,\s{G}_i) \quad \text{and} \quad H^{r}(k,\s{F}_{i-1}^{\D}) \to H^{r}(k,\s{F}_i^{\D}) \to H^{r}(k,\s{G}_i^{\D}) \]
	and Proposition \ref{vanishing graded} to conclude that the only value for $r$ away from $\{-1,0,1\}$ for which the groups $H^{2-r}(k,\s{F}_1)$ and $H^r(k,\s{F}_1^{\D})$ could be non trivial is $r=2$. Whence we conclude using the point $a)$ of Corollary \ref{tatecoro}.
\end{proof}


\subsection{Topology on the cohomology of the filtration pieces} \label{topology filtration}

In order to get a duality result for the filtration pieces $\s{F}_i$ of $\s{C}_X=\tau_{\leq 1}R\phi_*\G_{m,X}$, it is necessary to equip its cohomology groups with a suitable topology. Recall that the graded pieces $\s{G}_i$ are, up to a shift, sheaves represented by a smooth algebraic group (see \eqref{graded pieces}). We summarize the topology on the cohomology groups of these pieces in the following table:
\begin{equation*}
	\begin{tabular}{|c|c|c|c|}
		\hline
		& $\s{G}_0$		& $\s{G}_1$ & $\s{G}_2$ \\
		\hline
		$H^0$		& $k^*$	& $\{0\}$ & $\{0\}$ \\
		$H^1$		& $\{0\}$ & $\Pic_{X/k}^0(k)$ & discrete \\
		$H^2$		& discrete & discrete & finite discrete \\
		$H^3$		& $\{0\}$ & discrete & discrete \\
		\hline
	\end{tabular}
\end{equation*}
Since $\s{F}_0=\s{G}_0=\G_{m,k}$, we equip $H^i(k,\s{F}_0)$ with the same topology than $H^i(k,\G_{m,k})$. Now, the exact triangles \eqref{triangle1} and \eqref{triangle2} induce the following exact sequences of abelian groups
\begin{align} 
	H^n(k,\s{F}_{0}) \to H^n(k,\s{F}_1) \to H^n(k,\s{G}_1) \label{top seq1} \\
	H^n(k,\s{F}_{1}) \to H^n(k,\s{F}_2) \to H^n(k,\s{G}_2) \label{top seq2}
\end{align}
We will show that there exists a natural way to topologize the groups $H^n(k,\s{F}_i)$ such that \eqref{top seq1} and \eqref{top seq2} become strict exact sequences of abelian topological groups. We only have to topologize the groups $H^n(k,\s{F}_i)$ for $i=1$ or 2 and $n\in\{0,1,2,3\}$ since, by Corollary \ref{vanish}, for $n\notin\{0,1,2,3\}$ all of these groups are trivial. \\

\fbox{$i=1$} By the isomorphism $H^0(k,\s{F}_0) \overset{\sim}{\to} H^0(k,\s{F}_1)$ in Corollary \ref{tatecoro}, we equip $H^0(k,\s{F}_1)$ with the topology induced, via this isomorphism, from $H^0(k,\s{F}_0)=k^*$. For $n>1$, the groups $H^n(k,\s{F}_0) = H^n(k,\G_{m,k})$ and $H^n(k,\s{G}_1)\cong H^{n-1}(k,\Pic_{X/k}^0)$ have the discrete topology. As $H^n(k,\s{F}_1)$ is the middle term of the exact sequence \eqref{top seq1}
\[ H^n(k,\s{F}_{0}) \to H^n(k,\s{F}_1) \to H^n(k,\s{G}_1), \]
it will be naturally equipped with the discrete topology for $n>1$. For $n=1$, the sequence \eqref{top seq1} becomes
\[ H^1(k,\s{F}_{0}) \to H^1(k,\s{F}_1) \to H^1(k,\s{G}_1), \]
where $H^1(k,\s{F}_{0})=H^1(k,\G_{m,k})$ and $H^1(k,\s{G}_1)\cong \Pic_{X/k}^0(k)$. By Hilbert's 90 theorem, $H^1(k,\s{F}_1)$ is naturally equipped with the subspace topology from $\Pic_{X/k}^0(k)$. \\

\fbox{$i=2$} By the isomorphism $H^0(k,\s{F}_0) \overset{\sim}{\to} H^0(k,\s{F}_2)$ in Corollary \ref{tatecoro}, we equip $H^0(k,\s{F}_1)$ with the topology induced, via this isomorphism, from $H^0(k,\s{F}_0)=k^*$. Now, the groups $H^n(k,\s{F}_1)$ and $H^n(k,\s{G}_2)\cong H^{n-1}(k,\NS_{X/k})$ have the discrete topology for $n>1$. As $H^n(k,\s{F}_2)$ is the middle term of the exact sequence \eqref{top seq2}
\[ H^n(k,\s{F}_{1}) \to H^n(k,\s{F}_2) \to H^n(k,\s{G}_2), \]
it will be equipped with the discrete topology for $n>1$. For $n=1$, the sequence \eqref{top seq2} becomes
\[ 0 \to H^1(k,\s{F}_{1}) \to H^1(k,\s{F}_2) \to H^1(k,\s{G}_2), \]
where the zero at left comes is $H^0(k,\s{G}_2)\cong H^{-1}(k,\NS_{X/k})=0$. Since $H^1(k,\s{G}_2)\cong H^{0}(k,\NS_{X/k})$ has the discrete topology, there exists a unique topology on $H^1(k,\s{F}_2)$ such that the sequence above becomes strict exact (see \cite[Proposition 4.1, Corollary 2]{Calabi}). \\

In this way, we conclude the following proposition.

\begin{proposition} \label{top fil} 
	For every $i\in\{0,1,2\}$ and $n\in\Z$, there exists a unique topology on $H^n(k,\s{F}_i)$ such that the sequence
	\[ H^n(k,\s{F}_{i-1}) \to H^n(k,\s{F}_i) \to H^n(k,\s{G}_i), \]
	is a strict exact sequence of first countable abelian topological groups. More precisely,
	\begin{enumerate}[nosep, label=\alph*)]
		\item $H^0(k,\s{F}_i)$ is topologically isomorphic to $k^*$;
		\item $H^1(k,\s{F}_0)$ is trivial; $H^1(k,\s{F}_1)$ is a subspace of $\Pic_{X/k}^0(k)$; and $H^1(k,\s{F}_2)$ is an extension of a discrete group by $H^1(k,\s{F}_1)$;
		\item $H^n(k,\s{F}_i)$ are discrete for $n=2,3$. The group $H^3(k,\s{F}_0)$ is trivial; and 
		\item $H^n(k,\s{F}_i)$ is trivial for $n\notin\{0,1,2,3\}$.
	\end{enumerate}
	Thus, $H^n(k,\s{F}_i)$ is Hausdorff, locally compact, totally disconnected and second countable. 
\end{proposition}

\begin{proof}
	The topological properties of $H^n(k,\s{F}_i)$ follow from the previous construction and \cite[p. 58, Theorem 6.15]{Stroppel} once we prove that $H^1(k,\s{F}_1)$ verifies such properties. Note that we just have to verify local compacity on $H^1(k,\s{F}_1)$. In order to do that, it is sufficient to prove that $H^1(k,\s{F}_1)$ is open in $\Pic_{X/k}^0(k)$, which is true by Lemma \ref{boundary}.
\end{proof}

Similarly, we will treat the case of the Cartier dual of the filtration pieces. Recall that the dual graded pieces $\s{G}_i^{\D}$ are, up to a shift, either sheaves represented by a smooth algebraic group or a generalized 1-motive (see \eqref{dual graded pieces}). In the following table, we summarize the topology on the cohomology groups of the dual graded pieces:
\begin{equation*} 
	\begin{tabular}{|c|c|c|c|}
		\hline
		& $\s{G}_0^{\D}$		& $\s{G}_1^{\D}$ & $\s{G}_2^{\D}$ \\
		\hline
		$H^{-1}$	& $\{0\}$ & discrete & $\NS_{X/k}^D(k)$ \\
		$H^0$		& discrete	& $H^0(k,[T^D\to A^t])$ & finite discrete \\
		$H^1$		& $\{0\}$ & discrete & discrete \\
		$H^2$		& discrete & $\{0\}$ & $\{0\}$ \\
		\hline
	\end{tabular}
\end{equation*}
Since $\s{F}_0^\D=\s{G}_0^\D=\G_{m,k}^\D\cong\Z$, we equip $H^n(k,\s{F}_0^\D)$ with the discrete topology for all $n$. As we have done earlier, we will use the exact sequences
\begin{align} 
	H^n(k,\s{G}_1^{\D}) \to H^n(k,\s{F}_1^{\D}) \to H^n(k,\s{F}_{0}^{\D}) \label{dual top seq1} \\
	H^n(k,\s{G}_2^{\D}) \to H^n(k,\s{F}_2^{\D}) \to H^n(k,\s{F}_{1}^{\D}) \label{dual top seq2}
\end{align}
induced by the exact triangles \eqref{dual triangle1} and \eqref{dual triangle2}, to define a topology on $H^n(k,\s{F}_i^{\D})$. By Corollary \ref{vanish}, $H^n(k,\s{F}_i^{\D})$ is trivial for $i\in\{1,2\}$ and $n\notin\{-1,0,1,2\}$. Thus, \\

\fbox{$i=1$} By the isomorphism $H^2(k,\s{F}_1^{\D}) \overset{\sim}{\to} H^2(k,\s{F}_0^{\D})$ in Corollary \ref{tatecoro}, we equip $H^2(k,\s{F}_1^{\D})$ with the discrete topology. Now, the group $H^n(k,\s{F}_1^{\D})$ is the middle term of the exact sequence \eqref{dual top seq1}
\[ H^{n}(k,[T^D\to A^t]) \to H^n(k,\s{F}_1^{\D}) \to H^n(k,\Z). \]
Since $H^n(k,\Z)$ has the discrete topology for every $n$, there exists a unique topology on $H^n(k,\s{F}_1^{\D})$ such that the sequence above is strict exact. In particular, $H^n(k,\s{F}_1^{\D})$ will have discrete topology when $n=-1$ or 1 since, in these cases, $H^{n}(k,[T^D\to A^t])$ is also discrete. \\

\fbox{$i=2$} By the isomorphism $H^2(k,\s{F}_2^{\D}) \overset{\sim}{\to} H^2(k,\s{F}_0^{\D})$ in Corollary \ref{tatecoro}, we equip $H^2(k,\s{F}_1^{\D})$ with the discrete topology. Now, $H^n(k,\s{F}_2^{\D})$ is the middle term of the exact sequence \eqref{dual top seq2}
\[ H^{n+1}(k,\NS_{X/k}^D) \to H^n(k,\s{F}_2^{\D}) \to H^n(k,\s{F}_{1}^{\D}). \]
Note that for $n=-1$ or 1, the right side term has the discrete topology. Then there exists a unique topology on $H^n(k,\s{F}_2^{\D})$ such that the sequence above becomes strict exact for $n=-1$ or 1. For $n=0$, the situation is different and more delicate. The following exact triangle (see \eqref{cdfp})
\[ \s{F}_0 \to \s{F}_2 \to \Pic_{X/k}[-1] \to \s{F}_0[1] \]
yields the following exact sequence of abelian groups
\[ 0 \to H^1(k,\Pic_{X/k}^{\D}) \to H^0(k,\s{F}_2^{\D}) \to \Z. \]
Then, since $\Z$ is discrete, there exists a unique topology on $H^0(k,\s{F}_2^{\D})$ such that the exact sequence above becomes a strict exact sequence of abelian topological groups. Thus, the rows of the following commutative diagram of abelian groups
\begin{equation*}
	\begin{tikzcd}
		0 \ar{r} & H^1(k_{\sm},\Pic_{X/k}^{\D}) \ar{r} \ar{d}{f} & H^0(k,\s{F}_2^{\D}) \ar{r}{\deg} \ar{d}{g} & \Z \ar[d,equal] \\
		0 \ar{r} & H^0(k_{\sm},[T^D\to A^t]) \ar{r} & H^0(k_{\sm},\s{F}_1^{\D}) \ar{r} & \Z
	\end{tikzcd}
\end{equation*}
(see diagram \eqref{deg map 3}) are strict exact sequences of abelian topological groups. By Proposition \ref{top loc alg}, $f$ is continuous and, therefore, $g$ is continuous since $\Z$ is discrete (see \cite[Corollary 3.18]{topart}). Then, applying $(-)_{\Haus}$ to the diagram above yields the following commutative diagram of abelian groups
\begin{equation*}
	\begin{tikzcd}
		0 \ar{r} & H^1(k_{\sm},\Pic_{X/k}^{\D})_{\Haus} \ar{r} \ar{d}{f_{\Haus}} & H^0(k,\s{F}_2^{\D})_{\Haus} \ar{r}{\deg} \ar{d}{g_{\Haus}} & \Z \ar[d,equal] \\
		0 \ar{r} & H^0(k_{\sm},[T^D\to A^t])_{\Haus} \ar{r} & H^0(k_{\sm},\s{F}_1^{\D})_{\Haus} \ar{r} & \Z,
	\end{tikzcd}
\end{equation*}
whose rows are strict exact (see \cite[Lemma 2.13]{topart}) and $f_{\Haus}$ is strict continuous. Hence, by Corollary 3.18 in \cite{topart}, $g_{\Haus}$ is a strict continuous homomorphism of abelian topological groups. Thus,
\[ H^{1}(k,\NS_{X/k}^D) \to H^0(k,\s{F}_2^{\D})_{\Haus} \to H^0(k,\s{F}_{1}^{\D})_{\Haus}. \]
is a strict exact sequence of abelian topological groups. Further, by Proposition 3.20 in \cite{topart}, any pair of topologies on $H^0(k,\s{F}_2^{\D})$ satisfying that the sequence 
\[ H^{1}(k,\NS_{X/k}^D) \to H^0(k,\s{F}_2^{\D}) \to H^0(k,\s{F}_{1}^{\D}). \]
is strict exact, have isomorphic Hausdorff quotients. \\

Similar to Proposition \ref{top fil}, we obtain the following result.

\begin{proposition} \label{top dual fil} 
	For every $i\in\{0,1,2\}$ and $n\in\Z$, there exists a topology on $H^n(k,\s{F}_i^{\D})$ such that the sequence
	\[ H^n(k,\s{G}_{i}^{\D}) \to H^n(k,\s{F}_i^{\D}) \to H^n(k,\s{F}_{i-1}^{\D}), \]
	is a strict exact sequence of first countable abelian topological groups. These topologies are unique except possibly for $H^0(k,\s{F}_2^{\D})$, in which case the separation $H^0(k,\s{F}_2^{\D})_{\Haus}$ of $H^0(k,\s{F}_2^{\D})$ is uniquely determined. More precisely,
	\begin{enumerate}[nosep, label=\alph*)]
		\item $H^{-1}(k,\s{F}^{\D}_i)$ has discrete topology and for $i=0$ this group is trivial;
		\item $H^{0}(k,\s{F}^{\D}_0)$ has the discrete topology; the group $H^{0}(k,\s{F}^{\D}_1)$ is an extension of a discrete group by the (possibly non-Hausdorff) group $H^0(k,[T^D\to A^t])$; and the group $H^{0}(k,\s{F}^{\D}_2)$ is and extension of a subgroup of $H^{0}(k,\s{F}^{\D}_1)$ by a discrete finite group.
		\item $H^{n}(k,\s{F}^{\D}_i)$ has discrete topology for every $n\in\{1,2\}$. The group $H^{1}(k,\s{F}^{\D}_0)$ is trivial; and
		\item $H^{n}(k,\s{F}^{\D}_i)$ is trivial for $n\notin\{-1,0,1,2\}$.
	\end{enumerate}
	Thus, $H^n(k,\s{F}_i^{\D})$ is locally compact; and it is Hausdorff, except possibly $H^{0}(k,\s{F}_1^{\D})$ and $H^{0}(k,\s{F}_2^{\D})$.
\end{proposition}

\begin{corollary} \label{top dual fil haus}
	We have the following topological extensions of Hausdorff locally compact abelian topological groups
	\begin{gather*}
		0 \to H^0(k,[T^D\to A^t])_{\Haus} \to H^0(k,\s{F}_1^{\D})_{\Haus} \to \im(H^0(k,\s{F}_1^{\D}) \to H^0(k,\s{F}_0^{\D})) \to 0 \\
		0 \to H^1(k,\NS_{X/k}^D) \to H^0(k,\s{F}_2^{\D})_{\Haus} \to \im(H^0(k,\s{F}_2^{\D}) \to H^0(k,\s{F}_1^{\D}))_{\Haus} \to 0.
	\end{gather*}
	In particular, for every $i\in\{0,1,2\}$ and $n\in\Z$, $H^n(k,\s{F}_i^{\D})_{\Haus}$ is Hausdorff, locally compact, totally disconnected, second countable and compactly generated. 
\end{corollary}

\begin{proof}
	The two topological extensions above follow from Proposition \ref{top dual fil} and Lemma 2.13 in \cite{topart}. For the assertion about the topological properties of $H^n(k,\s{F}_i^{\D})_{\Haus}$, we only have to prove the second countability by Proposition \ref{top dual fil}. This follows from \cite[Proposition 2.C.8(2)]{Cornulier}\footnote{In \cite{Cornulier}, all topological groups are assumed to be Hausdorff.}.
\end{proof}


\subsection{Continuity of the boundary maps}

We start determining which cohomology groups of the filtration pieces are torsion. Recall that an abelian torsion group is \emph{of cofinite type} if its $n$-torsion subgroup is finite for all $n\in\N$. The following lemma recalls some properties about torsion groups.

\begin{lemma} \label{cofinite}
	Let $0 \to A \overset{f}{\to} B \overset{g}{\to} C$ be an exact sequence of abelian groups. Then, $B$ is torsion (resp. of cofinite type) whenever both $A$ and $C$ are torsion (resp. of cofinite type). In particular, if a morphism of torsion abelian groups $B\to C$ has finite kernel and $C$ is of cofinite type, then $B$ is of cofinite type too.
\end{lemma}

\begin{proof}
	Clearly if $A$ and $C$ are torsion, then so is $B$. Now, let us suppose that $A$ and $C$ are also of cofinite type. We have the following exact sequence
	\begin{equation} \label{torsion sequence}
		\ker(g)\cap B[n] \to B[n] \to C[n].
	\end{equation} 
	From the exactness of the main exact sequence we have that
	\[ \ker(g)\cap B[n]=\im(f)\cap B[n] \subseteq \im(f)[n]. \]
	Thus for concluding the proof we only have to note that $\im(f)[n]$ is finite (see the exact sequence \eqref{torsion sequence}). Let $b$ be in $\im(f)[n]$. Then there exists $a\in A$ such that $f(a)=b$ and $f(na)=nf(a)=0$. Since $f$ is injective, we conclude that $a\in A[n]$ and therefore $\im(f)[n]\subseteq f(A[n])$. Hence $\im(f)[n]$ is finite since so is $A[n]$.   
\end{proof}

Using the previous lemma we can determine which cohomology groups of the filtration pieces are torsion.

\begin{proposition} \label{torsion}
	Let $i\in\{0,1,2\}$. Then,
	\begin{enumerate}[nosep, label=\alph*)]
		\item $H^r(k,\s{F}_i)$ is torsion for $r\in\{2,3\}$.
		\item $H^r(k,\s{F}_i^{\D})$ is torsion for $r\in\{1,2\}$. 
	\end{enumerate}
\end{proposition}

\begin{proof}
	\noindent
	\begin{enumerate}[nosep, label=\alph*)]
		\item Let $r\in\{2,3\}$. We have the following exact sequence
		\[ H^r(k,\s{F}_{i-1}) \to H^r(k,\s{F}_{i}) \to H^r(k,\s{G}_{i}). \]
		By Proposition \ref{tate} we know that $H^r(k,\s{G}_i)$ is torsion for $i\in\{0,1,2\}$. Then, by Lemma \ref{cofinite} and an inductive argument using the fact that $\s{F}_0=\s{G}_0$, we conclude the proof.

		\item Similarly, for $r\in\{1,2\}$, using the exact sequence
		\[ H^r(k,\s{G}_i^{\D}) \to H^r(k,\s{F}_{i}^{\D}) \to H^r(k,\s{F}_{i-1}^{\D}), \]
		Lemma \ref{cofinite} and the fact that $H^r(k,\s{F}_0^{\D})\cong H^r(k,\Z)$ is torsion, we only have to prove that $H^r(k,\s{G}_i^{\D})$ is torsion for all $i\in\{0,1,2\}$ and $r\in\{1,2\}$. This last assertion follows from Proposition \ref{tate}. 
	\end{enumerate}
\end{proof}

The following lemma will allow us to take profinite completion without disrupting exactness in \eqref{pairing}.

\begin{lemma} \label{boundary}
	The boundary maps of \eqref{pairing} have finite images. More precisely,
	\begin{enumerate}[nosep, label=\alph*)]
		\item $H^r(k,\s{G}_i)\to H^{r+1}(k,\s{F}_{i-1})$ has finite image for $(r,i)\in\{(1,1),(1,2),(2,2)\}$;
		\item $H^r(k,\s{F}_{i-1}^{\D})\to H^{r+1}(k,\s{G}_{i}^{\D})$ has finite image for $(r,i)\in\{(0,1),(0,2),(-1,2)\}$.
	\end{enumerate}
	In any other case, boundary maps are zero.
\end{lemma}

\begin{proof}
	The vanishing of boundary maps away from the selected cases is directly obtained from the fact that in these cases either the source or the target of the map are zero maps (see Proposition \ref{tate} and Corollary \ref{vanish}).
	\begin{enumerate}[nosep, label=\alph*)]
		\item For $(r,i)=(2,2)$ the source of the boundary map is $H^2(k,\s{G}_2)=H^1(k,\NS_{X/k})$, which is finite by Proposition \ref{tate}, therefore its image is also finite. For $(r,i)=(1,2)$ the source of the boundary map is the finitely generated group $H^1(k,\s{G}_2)\cong \NS_{X/k}(k)$ and the target is $H^2(k,\s{F}_1)$, which is torsion by Corollary \ref{torsion}, then its image is finite. For $(r,i)=(1,1)$, we have the following morphism of exact triangles
		\[ \begin{tikzcd}
			\s{F}_0 \ar[r] \ar[d,equal] & \s{F}_1 \ar[r] \ar[d] & \s{G}_1 \ar[r] \ar[d] & \s{F}_0[1] \ar[d,equal] \\
			\G_{m,k} \ar[r] & \tau_{\leq 1}R\phi_*\G_{m,k} \ar[r] & \Pic_{X/k}[-1] \ar[r] & \G_{m,k}[1],
		\end{tikzcd} \]
		which induces the following commutative exact diagram
		\[ \begin{tikzcd}
			H^1(k,\s{G}_1) \ar[r] \ar[d] & \Br(k) \ar[d,equal] & \\
			\Pic_{X/k}(k) \ar[r] & \Br(k) \ar[r] & H^2(k,\tau_{\leq 1}R\phi_*\G_{m,X})
		\end{tikzcd} \]
		Note that $\Br(k) \to H^2(k,\tau_{\leq 1}R\phi_*\G_{m,X})$ factors through the restriction map $\Br(k)\to\Br(X)$ via the standard filtration $\tau_{\leq 0}R\phi_*\G_{m,X} \to \tau_{\leq 1}R\phi_*\G_{m,X} \to \tau_{\leq 2}R\phi_*\G_{m,X}$. Then, the image of the boundary map $H^1(k,\s{G}_1) \to \Br(k)$ is contained in the kernel of $\Br(k)\to\Br(X)$, which is finite. A proof of this last assertion follows similarly as \cite[Proposition 5.4.4]{BG}) using the fact that the group of rational points of an abelian variety over a $p$-field is topologically finitely generated.
		
		\item For $(r,i)=(-1,2)$ the target of the boundary map is $H^0(k,\s{G}_2^{\D})=H^1(k,\NS_{X/k}^D)$, which is finite by Proposition \ref{tate}. For $(r,i)=(0,1)$, the source of the boundary map is $H^0(k,\s{F}_0^{\D})=\Z$ and the target is $H^1(k,\s{G}_1^{\D})$, which is torsion by Proposition \ref{tate}, therefore its image is finite. For $(r,i)=(0,2)$, the boundary map 
		\[ \partial : H^1(k,\s{G}_2) \to H^2(k,\s{F}_1) \]
		is strict continuous since its source and target are discrete; and it has finite image by the previous point. Then $\partial$ factors as
		\[ H^1(k,\s{G}_2) \to H^1(k,\s{G}_2)^\wedge \overset{\hat{\partial}}{\to} H^2(k,\s{F}_1) \]
		by Lemma \ref{profinite}. Thus, from the Yoneda pairing, we get the following commutative diagram
		\begin{equation} \label{square}
			\begin{split}
				\xymatrix{
					H^0(k,\s{F}_1^{\D}) \ar[r] \ar[d] & H^1(k,\s{G}_2^{\D}) \ar[d]^{\sim} \\
					\Hom(H^2(k,\s{F}_1),\Q/\Z) \ar[r] & \Hom(H^1(k,\s{G}_2)^\wedge,\Q/\Z),
				}
			\end{split}
		\end{equation}
		where the right hand vertical arrow of \eqref{square} is an isomorphism by Proposition \ref{tate}. Moreover the bottom arrow in \eqref{square} is continuous since $\hat{\partial}$ is continuous. On the other hand, the group $H^2(k,\s{F}_1)$ is discrete by Proposition \ref{top fil} and torsion by Corollary \ref{torsion}, thereby the group $\Hom(H^2(k,\s{F}_1),\Q/\Z)$ is a profinite abelian group by Pontryagin duality. Similarly, we have that $\Hom(H^1(k,\s{G}_2)^\wedge,\Q/\Z)$ is a discrete abelian torsion group since $H^1(k,\s{G}_2)^\wedge$ is profinite abelian group. Then the image of the bottom arrow in \eqref{square} is finite since it is compact and discrete. Hence the boundary map $H^0(k,\s{F}_1^{\D}) \to H^1(k,\s{G}_2^{\D})$ has finite image.
	\end{enumerate}
\end{proof}

In the following result we conclude the strict continuity of the sequences in \eqref{pairing}.

\begin{corollary} \label{boundary cont}
	All boundary maps in \eqref{pairing} are continuous and strict.
\end{corollary}

\begin{proof}
	We note that in the selected cases of Lemma \ref{boundary} all the boundary maps have discrete target and the only ones whose source is not discrete are the cases $(1,1)$ of $a)$ and $(0,2)$ of $b)$. Hence, away from these cases, the boundary maps are strict continuous. Thus we only have to check the continuity of the boundary maps $H^1(k,\s{G}_1)\to H^{2}(k,\s{F}_{0})$ and $H^0(k,\s{F}_{1}^{\D})\to H^{1}(k,\s{G}_{2}^{\D})$, that is, we have to prove that the kernels of these maps are open since the target of these maps are discrete by the Propositions \ref{top fil} and \ref{top dual fil}. By Lemma \ref{boundary}, the kernels of the maps 
	\[ H^1(k,\s{G}_1)\to H^{2}(k,\s{F}_{0}) \quad \text{and} \quad H^0(k,\s{F}_{1}^{\D})\to H^{1}(k,\s{G}_{2}^{\D}) \]
	have finite index in $H^1(k,\s{G}_1)$ and $H^0(k,\s{F}_{1}^{\D})$ respectively. On one hand, the group $H^1(k,\s{G}_1)$ is isomorphic to $\Pic_{X/k}^0(k)$. Since $\Pic_{X/k}^0$ is a semiabelian variety over a $p$-adic field, all its finite index subgroup are open (see \cite[Example 3.11]{topart}). Hence, the boundary map $H^1(k,\s{G}_1)\to H^{2}(k,\s{F}_{0})$ is continuous. On the other hand, the group $H^0(k,\s{F}_{1}^{\D})$ is a topological extension of a discrete group by $H^0(k,[T^D\to A^t])$. Hence, since every finite index  subgroup of $H^0(k,[T^D\to A^t])$ is open (see \cite[Remark 2.4]{HS05}), every finite index  subgroup of $H^0(k,\s{F}_{1}^{\D})$ is open (see Proposition 3.10 \emph{loc. cit.}). In particular, $\ker(H^0(k,\s{F}_{1}^{\D})\to H^{1}(k,\s{G}_{2}^{\D}))$ is open in $H^0(k,\s{F}_{1}^{\D})$ and, therefore, the boundary map $H^0(k,\s{F}_{1}^{\D})\to H^{1}(k,\s{G}_{2}^{\D})$ is continuous.
\end{proof}

\begin{corollary} 
	The map $H^0(k,\s{F}_{2}^{\D})\to H^0(k,\s{F}_{1}^{\D})$ is open.
\end{corollary}

\begin{proof}
	By Corollary \ref{boundary cont}, the boundary map $\partial:H^0(k,\s{F}_{1}^{\D})\to H^{1}(k,\s{G}_{2}^{\D})$ is continuous. Hence the image of $H^0(k,\s{F}_{2}^{\D})\to H^0(k,\s{F}_{1}^{\D})$ is open since $H^{1}(k,\s{G}_{2}^{\D})$ is discrete. Hence $\partial$ is open since $\partial$ it is also strict. 
\end{proof}


\subsection{Tate duality for the filtration pieces}

In this section we will prove that the filtration pieces satisfy a duality result as in Proposition \ref{tate}. In order to do that, it is necessary to do the following two topological operations: Pontryagin dual and profinite completion. Thus, we must ensure that when applying Pontryagin dual and profinite completion to certain exact sequences the exactness will not be disrupted. \\

Throughout this section, we will use the following two commutative diagrams of abelian groups
\begin{equation} \label{case1}
	\begin{tikzcd}[column sep=1.5em]
		H^{1-r}(k,\s{G}_i) \ar{r}{\partial} \ar{d}{\alpha} & H^{2-r}(k,\s{F}_{i-1}) \ar{r} \ar{d}{\beta} & H^{2-r}(k,\s{F}_i) \ar{r} \ar{d}{\gamma} & H^{2-r}(k,\s{G}_i) \ar{r}{\partial} \ar{d}{\delta} & H^{3-r}(k,\s{F}_{i-1}) \ar{d}{\epsilon} \\
		H^{r+1}(k,\s{G}_i^{\D})^* \ar{r} & H^r(k,\s{F}_{i-1}^{\D})^* \ar{r} & H^r(k,\s{F}_{i}^{\D})^* \ar{r} & H^r(k,\s{G}_{i}^{\D})^* \ar{r} & H^{r-1}(k,\s{F}_{i-1}^{\D})^* 
	\end{tikzcd}
\end{equation}
and
\begin{equation} \label{case2}
	\begin{tikzcd}[column sep=1.5em]
		H^{r-1}(k,\s{F}_{i-1}^{\D}) \ar{r}{\partial'} \ar{d}{\alpha'} & H^r(k,\s{G}_{i}^{\D}) \ar{r} \ar{d}{\beta'} & H^r(k,\s{F}_{i}^{\D}) \ar{r} \ar{d}{\gamma'} & H^r(k,\s{F}_{i-1}^{\D}) \ar{r}{\partial'} \ar{d}{\delta'} & H^{r+1}(k,\s{G}_i^{\D}) \ar{d}{\epsilon'} \\
		H^{3-r}(k,\s{F}_{i-1})^* \ar{r} & H^{2-r}(k,\s{G}_i)^* \ar{r} & H^{2-r}(k,\s{F}_i)^* \ar{r} & H^{2-r}(k,\s{F}_{i-1})^* \ar{r} & H^{1-r}(k,\s{G}_i)^*,
	\end{tikzcd}
\end{equation}
where all groups above are regarded as abelian topological groups via the assignment made in Section \ref{topology filtration}. In the case of $H^0(k,\s{F}^{\D}_2)$, we have to choose a topology, say $\tau$ (see Proposition \ref{top dual fil}). Thus, all rows of diagrams \eqref{case1} and \eqref{case2} are strict exact sequences of abelian topological groups.

\begin{proposition} \label{continuity yoneda}
	The Yoneda pairing for the filtration pieces $\s{F}_i$
	\begin{equation} \label{yoneda pairing}
		H^r(k,\s{F}_i^{\D}) \times H^{2-r}(k,\s{F}_i) \to H^2(k,\G_{m,k})\cong\Q/\Z
	\end{equation}
	is continuous for every $i\in\{0,1,2\}$ and $r\in\Z$.
\end{proposition}

\begin{proof}
	Note that for $r\notin\{-1,0,1,2,3\}$, the left-side of the pairing \eqref{yoneda pairing} is trivial. Moreover, by Propositions \ref{top fil} and \ref{top dual fil}, in \eqref{yoneda pairing} always appears a discrete group. Furthermore, for $r=-1,2$ or 3, both groups in the left-side of \eqref{yoneda pairing} are discrete. Then, by Proposition \ref{pairing continuous}, we only have to check the continuity of the middle vertical map $\gamma:H^{2-r}(k,\s{F}_i)\to H^r(k,\s{F}_{i}^{\D})^*$ in \eqref{case1} or the middle vertical map $\gamma':H^r(k,\s{F}_{i}^{\D})\to H^{2-r}(k,\s{F}_i)^*$ in \eqref{case2} as appropriate and only when $r=0,1$. We will prove this result inductively using the diagrams \eqref{case1} and \eqref{case2}. Thus we may assume that, in theses diagrams, the second vertical maps $\beta$ and $\beta'$ and the fourth vertical maps $\delta$ and $\delta'$ are continuous.
	
	For $r=0$, we have to check the continuity of $\gamma':H^0(k,\s{F}_i^{\D}) \to H^2(k,\s{F}_i)^*$, which is automatically continuous for $i=0$ since, in this case, $H^0(k,\s{F}_i^{\D})$ is discrete. For $i=1$, the result follows from point $a)$ in Proposition 3.22 in \cite{topart} and the fact that $H^0(k,\s{F}_0^{\D})$ is discrete. For $i=2$, the result follows from point $a)$ in Proposition 3.22 in \cite{topart}, Corollary \ref{top dual fil haus} and the fact that both $H^0(k,\s{G}_2^{\D})$ and $H^2(k,\s{G}_2)$ are discrete and finite.
	
	For $r=1$, we have to check the continuity of $\gamma:H^1(k,\s{F}_i)\to H^1(k,\s{F}_i^{\D})^*$, which is automatically continuous for $i=0$ since, in this case, both the source and target of $\gamma$ are trivial. For $i=1$, the result follows from the continuity of $\delta:H^1(k,\s{G}_1)\to H^1(k,\s{G}_1^{\D})^*$ (see Proposition \ref{tate}) and the triviality of the groups $H^1(k,\s{F}_0)$ and $H^1(k,\s{F}_0^{\D})$. For $i=2$, the result follows from point $a)$ in Proposition 3.22 in \cite{topart} and the fact that $H^1(k,\s{G}_2)$ is discrete. 
\end{proof}

\begin{proposition} \label{non-degeneracy}
	The Yoneda pairing for the filtration pieces $\s{F}_i$
	\[ H^r(k,\s{F}_i^{\D}) \times H^{2-r}(k,\s{F}_i) \to H^2(k,\G_{m,k})\cong\Q/\Z \]
	is non-degenerate for every $i\in\{0,1,2\}$ and $r\in\Z$. Furthermore, these cohomology groups are trivial for $r\notin\{-1,0,1,2,3\}$.
\end{proposition}

\begin{proof}	
	We start proving that the middle vertical maps $\gamma:H^{2-r}(k,\s{F}_i)\to H^r(k,\s{F}_{i}^{\D})^*$ and $\gamma':H^r(k,\s{F}_{i}^{\D})\to H^{2-r}(k,\s{F}_i)^*$ in \eqref{case1} and \eqref{case2}, respectively, are injective. By Proposition \ref{tate}, the Yoneda pairing for $\s{F}_0=\s{G}_0=\G_m$ is non-degenerate. Hence, using a recursive argument, we may assume that all vertical arrows but the middle one of \eqref{case1} and \eqref{case2} are injective. Observe that, since all boundary maps have finite image (see Lemma \ref{boundary}), we may replace $\alpha$ (resp. $\alpha'$) by $\alpha^\wedge$ (resp. $\alpha'^\wedge$) without disrupting neither the exactness nor the commutativity, whenever the target of $\alpha$ (resp. $\alpha'$) is profinite (see Lemma \ref{profinite}). Thus, $\gamma$ will be injective under one of the following conditions:
	\begin{enumerate}[nosep, label=(\arabic*)]
		\item $\alpha$ (resp. $\alpha'$) is an isomorphism.
		\begin{equation} \label{case1alfa}
			\begin{tikzcd}[column sep=1.5em]
				H^{1-r}(k,\s{G}_i) \ar{r}{\partial} \ar{d}{\alpha} & H^{2-r}(k,\s{F}_{i-1}) \ar{r} \ar{d}{\beta} & H^{2-r}(k,\s{F}_i) \ar{r} \ar{d}{\gamma} & H^{2-r}(k,\s{G}_i) \ar{r}{\partial} \ar{d}{\delta} & H^{3-r}(k,\s{F}_{i-1}) \ar{d}{\epsilon} \\
				H^{r+1}(k,\s{G}_i^{\D})^* \ar{r} & H^r(k,\s{F}_{i-1}^{\D})^* \ar{r} & H^r(k,\s{F}_{i}^{\D})^* \ar{r} & H^r(k,\s{G}_{i}^{\D})^* \ar{r} & H^{r-1}(k,\s{F}_{i-1}^{\D})^* 
			\end{tikzcd}
		\end{equation}
		\begin{equation} \label{case2alfa}
			\begin{tikzcd}[column sep=1.5em]
				H^{r-1}(k,\s{F}_{i-1}^{\D}) \ar{r}{\partial'} \ar{d}{\alpha'} & H^r(k,\s{G}_{i}^{\D}) \ar{r} \ar{d}{\beta'} & H^r(k,\s{F}_{i}^{\D}) \ar{r} \ar{d}{\gamma'} & H^r(k,\s{F}_{i-1}^{\D}) \ar{r}{\partial'} \ar{d}{\delta'} & H^{r+1}(k,\s{G}_i^{\D}) \ar{d}{\epsilon'} \\
				H^{3-r}(k,\s{F}_{i-1})^* \ar{r} & H^{2-r}(k,\s{G}_i)^* \ar{r} & H^{2-r}(k,\s{F}_i)^* \ar{r} & H^{2-r}(k,\s{F}_{i-1})^* \ar{r} & H^{1-r}(k,\s{G}_i)^*,
			\end{tikzcd}
		\end{equation}
		\item $\alpha^\wedge$ (resp. $\alpha'^\wedge$) is an isomorphism when the target of $\alpha$ (resp. $\alpha'$) is profinite.
		\begin{equation} \label{case1pro}
			\begin{tikzcd}[column sep=1.5em]
				H^{1-r}(k,\s{G}_i)^\wedge \ar{r}{\partial^\wedge} \ar{d}{\alpha^\wedge} & H^{2-r}(k,\s{F}_{i-1}) \ar{r} \ar{d}{\beta} & H^{2-r}(k,\s{F}_i) \ar{r} \ar{d}{\gamma} & H^{2-r}(k,\s{G}_i) \ar{r}{\partial} \ar{d}{\delta} & H^{3-r}(k,\s{F}_{i-1}) \ar{d}{\epsilon} \\
				H^{r+1}(k,\s{G}_i^{\D})^* \ar{r} & H^r(k,\s{F}_{i-1}^{\D})^* \ar{r} & H^r(k,\s{F}_{i}^{\D})^* \ar{r} & H^r(k,\s{G}_{i}^{\D})^* \ar{r} & H^{r-1}(k,\s{F}_{i-1}^{\D})^* 
			\end{tikzcd}
		\end{equation}
		\begin{equation} \label{case2pro}
			\begin{tikzcd}[column sep=1.5em]
				H^{r-1}(k,\s{F}_{i-1}^{\D})^\wedge \ar{r}{\partial'^\wedge} \ar{d}{\alpha'^\wedge} & H^r(k,\s{G}_{i}^{\D}) \ar{r} \ar{d}{\beta'} & H^r(k,\s{F}_{i}^{\D}) \ar{r} \ar{d}{\gamma'} & H^r(k,\s{F}_{i-1}^{\D}) \ar{r}{\partial'} \ar{d}{\delta'} & H^{r+1}(k,\s{G}_i^{\D}) \ar{d}{\epsilon'} \\
				H^{3-r}(k,\s{F}_{i-1})^* \ar{r} & H^{2-r}(k,\s{G}_i)^* \ar{r} & H^{2-r}(k,\s{F}_i)^* \ar{r} & H^{2-r}(k,\s{F}_{i-1})^* \ar{r} & H^{1-r}(k,\s{G}_i)^*,
			\end{tikzcd}
		\end{equation}
		\item $\beta$ (resp. $\beta'$) has trivial source and target.
	\end{enumerate}
	
	For $i=1$, on one hand, in diagram \eqref{case1alfa}, for $r=-1,1,3$ and $r=2$ the source and target of $\beta$ and $\alpha$, respectively, are trivial and, therefore, condition $(3)$ and $(1)$ are satisfied, respectively. For $r=0$, in diagram \eqref{case1pro}, condition $(2)$ is satisfied by Proposition \ref{tate}. Hence, $\gamma$ is injective in this case. On the other hand, in diagram \eqref{case2alfa}, for $r=-1,0,2$ and 3, the source and target of $\alpha'$ and $\beta'$ respectively are trivial and, therefore, conditions $(1)$ and $(3)$ are satisfied, respectively. For $r=1$, in diagram \eqref{case2pro}, condition $(2)$ is satisfied by Proposition \ref{tate}. Hence, $\gamma'$ is injective in this case.
	
	For $i=2$, on one hand, in diagram \eqref{case1alfa}, for $r=1,2,3$ the condition $(1)$ is satisfied since the source and target of $\alpha$ are trivial. For $r=-1$, $\alpha$ is an isomorphism by Proposition \ref{tate}. For $r=0$, in diagram \eqref{case1pro}, the condition $(2)$ is satisfied by Proposition \ref{tate}. Hence, $\gamma$ is injective in this case. On the other hand, in diagram \eqref{case2alfa}, for $r=2,3$ and $-1$ the source and target of $\beta'$ and $\alpha'$, respectively, are trivial (see Corollary \ref{tatecoro}) and, therefore, conditions $(3)$ and $(1)$ are satisfied, respectively. For $r=0$, in diagram \eqref{case1pro}, condition $(2)$ is satisfied by Corollary \ref{tatecoro} and Proposition \ref{tate}. For $r=1$, in diagram \eqref{case2pro}, condition $(2)$ is satisfied by the point $a)$ of Proposition \ref{perfect pairing 1}, whose proof will be given below.
	
	For $r\notin\{-1,0,1,2,3\}$ the cohomology groups in the pairing are trivial by Lemma \ref{tatecoro}. In this way, we deduce the non-degenerateness of the pairing.
\end{proof}

\begin{proposition} \label{perfect pairing 1}
	The Yoneda pairing for the filtration pieces $\s{F}_i$
	\[ H^r(k,\s{F}_i^{\D}) \times H^{2-r}(k,\s{F}_i) \to H^2(k,\G_{m,k})\cong\Q/\Z \]
	induces continuous perfect pairings between the following profinite abelian groups and discrete abelian torsion groups, respectively:
	\begin{enumerate}[nosep, label=\alph*)]
		\item $H^0(k,\s{F}_i^{\D})^\wedge$ and $H^{2}(k,\s{F}_i)$;
		\item $H^{-1}(k,\s{F}_i^{\D})^\wedge$ and $H^{3}(k,\s{F}_i)$.
	\end{enumerate}
\end{proposition}

\begin{proof}
	By Proposition \ref{continuity yoneda}, the Yoneda pairing 
	\[ H^r(k,\s{F}_i^{\D}) \times H^{2-r}(k,\s{F}_i) \to \Q/\Z \]
	is continuous. Moreover, in this pairing, there is always a discrete group, whence we obtain the pairings $a)$ and $b)$. Now, we will prove the perfection of these new pairings.
	\begin{enumerate}[nosep, label=\alph*)]
		\item We have to check the perfection of the pairing
		\[ H^0(k,\s{F}_i^{\D})^\wedge \times H^{2}(k,\s{F}_i) \to \Q/\Z, \]
		induced by the Yoneda pairing. In order to do that, we have to check that 
		\begin{enumerate}[nosep, label=\roman*.]
			\item $\gamma:H^2(k,\s{F}_i)\to H^{0}(k,\s{F}_i^{\D})^*$, in diagram \eqref{case1} for $r=0$, is bijective; and
			\item $\gamma'^\wedge: H^{0}(k,\s{F}_i^{\D})^\wedge \to H^2(k,\s{F}_i)^*$, induced by $\gamma'$ in diagram \eqref{case2} for $r=0$, is bijective.
		\end{enumerate}
		For $i=0$, the perfection of the pairing follows from Proposition \ref{tate}.
		
		For $i=1$, diagram \eqref{case2} becomes the following commutative diagram of abelian groups
		\begin{equation} \label{case02 i1}
			\begin{tikzcd}[column sep=1.5em]
				0 \ar{r} & H^0(k,\s{G}_{1}^{\D}) \ar{r} \ar{d}{\beta'} & H^0(k,\s{F}_{1}^{\D}) \ar{r} \ar{d}{\gamma'} & H^0(k,\s{F}_{0}^{\D}) \ar{r}{\partial'} \ar{d}{\delta'} & H^{1}(k,\s{G}_1^{\D}) \ar{d}{\epsilon'} \\
				0 \ar{r} & H^{2}(k,\s{G}_1)^* \ar{r} & H^{2}(k,\s{F}_1)^* \ar{r} & H^{2}(k,\s{F}_{0})^* \ar{r} & H^{1}(k,\s{G}_1)^*.
			\end{tikzcd}
		\end{equation}
		whose rows are strict exact. Moreover, the image of $\partial'$ is finite by Lemma \ref{boundary}, and $H^0(k,\s{F}_{0}^{\D})\cong\Z$ is discrete and finitely generated. Hence, by Lemma \ref{profinite} and Remark \ref{exactness profinite}, diagram \eqref{case02 i1} induces the following commutative diagram of abelian groups
		\begin{equation*}
			\begin{tikzcd}[column sep=1.5em]
				0 \ar{r} & H^0(k,\s{G}_{1}^{\D})^\wedge \ar{r} \ar{d}{\beta'^\wedge} & H^0(k,\s{F}_{1}^{\D})^\wedge \ar{r} \ar{d}{\gamma'^\wedge} & H^0(k,\s{F}_{0}^{\D})^\wedge \ar{r}{\partial'^\wedge} \ar{d}{\delta'^\wedge} & H^{1}(k,\s{G}_1^{\D}) \ar{d}{\epsilon'} \\
				0 \ar{r} & H^{2}(k,\s{G}_1)^* \ar{r} & H^{2}(k,\s{F}_1)^* \ar{r} & H^{2}(k,\s{F}_{0})^* \ar{r} & H^{1}(k,\s{G}_1)^*.
			\end{tikzcd}
		\end{equation*}
		whose rows are strict exact. By Proposition \ref{tate}, $\beta'^\wedge,\delta'^\wedge$ and $\epsilon'$ are topological isomorphisms. Hence, the map $\gamma'^\wedge$ is a topological isomorphism (see \cite[Theorem 4.1]{topart}). Finally, diagram \eqref{case1} gives us the following commutative diagram of abelian groups
		\begin{equation*}
			\begin{tikzcd}[column sep=1.5em]
				H^1(k,\s{G}_1) \ar{r}{\partial} \ar{d}{\alpha} & H^{2}(k,\s{F}_{0}) \ar{r} \ar{d}{\beta} & H^{2}(k,\s{F}_1) \ar{r} \ar{d}{\gamma} & H^{2}(k,\s{G}_1) \ar{r} \ar{d}{\delta} & 0 \\
				H^1(k,\s{G}_1^{\D})^* \ar{r} & H^0(k,\s{F}_{0}^{\D})^* \ar{r} & H^0(k,\s{F}_{1}^{\D})^* \ar{r} & H^0(k,\s{G}_{1}^{\D})^* \ar{r} & 0,
			\end{tikzcd}
		\end{equation*}
		whose rows are strict exact. The image of $\partial'$ is finite by Lemma \ref{boundary}. Then, by Lemma \ref{profinite}, we can change $\alpha$ by $\alpha^\wedge$ in the diagram above and obtain the following commutative diagram of abelian groups
		\begin{equation*}
			\begin{tikzcd}[column sep=1.5em]
				H^1(k,\s{G}_1)^\wedge \ar{r}{\partial} \ar{d}{\alpha^\wedge} & H^{2}(k,\s{F}_{0}) \ar{r} \ar{d}{\beta} & H^{2}(k,\s{F}_1) \ar{r} \ar{d}{\gamma} & H^{2}(k,\s{G}_1) \ar{r} \ar{d}{\delta} & 0 \\
				H^1(k,\s{G}_1^{\D})^* \ar{r} & H^0(k,\s{F}_{0}^{\D})^* \ar{r} & H^0(k,\s{F}_{1}^{\D})^* \ar{r} & H^0(k,\s{G}_{1}^{\D})^* \ar{r} & 0,
			\end{tikzcd}
		\end{equation*}
		whose rows are strict exact. By Proposition \ref{tate}, $\alpha^\wedge,\ \beta$ and $\delta$ are isomorphisms and therefore $\gamma$ is an isomorphism. 
		
		For $i=2$, we have the following commutative diagram of abelian groups
		\begin{equation} \label{case02 i2} 
			\begin{tikzcd}[column sep=1.5em]
				H^{-1}(k,\s{F}_{1}^{\D}) \ar{r}{\partial'} \ar{d}{\alpha'} & H^0(k,\s{G}_{2}^{\D}) \ar{r} \ar{d}{\beta'} & H^0(k,\s{F}_{2}^{\D})_{\Haus} \ar{r} \ar{d}{\gamma'} & H^0(k,\s{F}_{1}^{\D})_{\Haus} \ar{r}{\partial'} \ar{d}{\delta'} & H^{1}(k,\s{G}_2^{\D}) \ar{d}{\epsilon'} \\
				H^{3}(k,\s{F}_{1})^* \ar{r} & H^{2}(k,\s{G}_2)^* \ar{r} & H^{2}(k,\s{F}_2)^* \ar{r} & H^{2}(k,\s{F}_{1})^* \ar{r} & H^{1}(k,\s{G}_2)^*.
			\end{tikzcd}
		\end{equation}
		whose rows are strict exact. Moreover, the image of $\partial'$ is finite by Lemma \ref{boundary}, $H^0(k,\s{G}_{2}^{\D})$ is finite and $H^{1}(k,\s{G}_2^{\D})$ is discrete and finitely generated. Hence, by Lemma \ref{profinite} and Remark \ref{exactness profinite}, the diagram \eqref{case02 i2} induces the following commutative diagram of abelian groups
		\begin{equation*}  
			\begin{tikzcd}[column sep=1.5em]
				H^{-1}(k,\s{F}_{1}^{\D})^\wedge \ar{r}{\partial'} \ar{d}{\alpha'^\wedge} & H^0(k,\s{G}_{2}^{\D}) \ar{r} \ar{d}{\beta'} & H^0(k,\s{F}_{2}^{\D})^\wedge \ar{r} \ar{d}{\gamma'^\wedge} & H^0(k,\s{F}_{1}^{\D})^\wedge \ar{r}{\partial'} \ar{d}{\delta'^\wedge} & H^{1}(k,\s{G}_2^{\D}) \ar{d}{\epsilon'} \\
				H^{3}(k,\s{F}_{1})^* \ar{r} & H^{2}(k,\s{G}_2)^* \ar{r} & H^{2}(k,\s{F}_2)^* \ar{r} & H^{2}(k,\s{F}_{1})^* \ar{r} & H^{1}(k,\s{G}_2)^*.
			\end{tikzcd}
		\end{equation*}
		whose rows are strict exact. By Proposition \ref{tate} and the previous case (for $i=1$), $\alpha'^\wedge,\ \beta',\ \delta'^\wedge$ and $\epsilon'^\wedge$ are topological isomorphisms. Hence, the map $\gamma'^\wedge$ is a topological isomorphism (see \cite[Theorem 4.1]{topart}). Finally, as in the previous case, we also conclude that the map $H^2(k,\s{F}_2) \to H^{0}(k,\s{F}_2^{\D})^*$ is an isomorphism.
		\item By Corollary \ref{tatecoro}, $H^{-1}(k,\s{F}_1^{\D})\cong H^{-1}(k,\s{G}_1^{\D})$ and $H^3(k,\s{G}_1)\cong H^3(k,\s{F}_1)$ for every $i\in\{0,1,2\}$. Hence, the pairing
		\[ H^{-1}(k,\s{F}_i^{\D}) \times H^3(k,\s{F}_i)^\wedge \to \Q/\Z, \]
		induced by the Yoneda pairing is continuous and perfect by Proposition \ref{tate}.
	\end{enumerate}
\end{proof}

\begin{proposition} \label{perfect pairing 2}
	The Yoneda pairing for the filtration pieces $\s{F}_i$
	\[ H^r(k,\s{F}_i^{\D}) \times H^{2-r}(k,\s{F}_i) \to H^2(k,\G_{m,k})\cong\Q/\Z \]
	induces continuous perfect pairings between the following discrete abelian torsion groups and profinite abelian groups, respectively:
	\begin{enumerate}[nosep, label=\alph*), start=3]
		\item $H^2(k,\s{F}_i^{\D})$ and $H^{0}(k,\s{F}_i)^\wedge$;
		\item $H^1(k,\s{F}_i^{\D})$ and $H^1(k,\s{F}_i)^{\wedge}$.
	\end{enumerate}
\end{proposition}

\begin{proof}
	By Proposition \ref{continuity yoneda}, the Yoneda pairing 
	\[ H^r(k,\s{F}_i^{\D}) \times H^{2-r}(k,\s{F}_i) \to \Q/\Z \]
	is continuous. Moreover, in this pairing, there is always a discrete group, whence we obtain the pairings $c)$ and $d)$. Now, we will prove the perfection of these new pairings.
	\begin{enumerate}[nosep, label=\alph*), start=3]
		\item By Corollary \ref{tatecoro}, $H^{0}(k,\s{F}_i)\cong H^0(k,\s{G}_0)$ and $H^{2}(k,\s{F}_i^{\D})\cong H^2(k,\s{G}_0^{\D})$ for every $i\in\{0,1,2\}$. Hence, the pairing 
		\[ H^{2}(k,\s{F}_i^{\D}) \times H^0(k,\s{F}_i)^\wedge \to \Q/\Z, \]
		induced by the Yoneda pairing is continuous and perfect by Proposition \ref{tate}.
		\item We have to check the perfection of the pairing
		\[ H^1(k,\s{F}_i^{\D}) \times H^{1}(k,\s{F}_i)^\wedge \to \Q/\Z, \]
		induced by the Yoneda pairing. In order to do that, we have to check that 
		\begin{enumerate}[nosep, label=\roman*.]
			\item $\gamma:H^1(k,\s{F}_i)\to H^1(k,\s{F}_i^{\D})^*$, in diagram \eqref{case1} for $r=1$, is bijective; and
			\item $\gamma'^\wedge: H^1(k,\s{F}_i^{\D})^\wedge \to H^1(k,\s{F}_i)^*$, induced by $\gamma'$ in diagram \eqref{case2} for $r=1$, is bijective.
		\end{enumerate}
		
		For $i=0$ the map $H^{1}(k,\s{F}_0)^\wedge \to H^1(k,\s{F}_0^{\D})^*$ is an isomorphism trivially because both groups are trivial (Hilbert's 90 theorem and $H^1(k,\Z)=0$).
		
		For $i=1$, we have the commutative diagram of abelian groups
		\begin{equation*}
			\begin{tikzcd}[column sep=1.5em]
				0 \ar[r] & H^{1}(k,\s{F}_1) \ar[r] \ar{d}{\gamma} & H^{1}(k,\s{G}_1) \ar{r}{\partial} \ar{d}{\delta} & H^{2}(k,\s{F}_{0}) \ar{d}{\epsilon} \\
				0 \ar[r] & H^1(k,\s{F}_{1}^{\D})^* \ar[r] & H^1(k,\s{G}_{1}^{\D})^* \ar[r] & H^{0}(k,\s{F}_{0}^{\D})^*,
			\end{tikzcd}
		\end{equation*}
		whose rows are strict exact. By Lemma \ref{boundary}, the boundary map $\partial$ has finite image. Then, by Lemma \ref{profinite} and Remark \ref{exactness profinite}, we obtain a commutative diagram of abelian groups
		\begin{equation} \label{case11 i1}
			\begin{tikzcd}[column sep=1.5em]
				0 \ar[r] & H^{1}(k,\s{F}_1)^\wedge \ar[r] \ar{d}{\gamma^\wedge} & H^{1}(k,\s{G}_1)^\wedge \ar{r}{\partial^\wedge} \ar{d}{\delta^\wedge} & H^{2}(k,\s{F}_{0}) \ar{d}{\epsilon} \\
				0 \ar[r] & H^1(k,\s{F}_{1}^{\D})^* \ar[r] & H^1(k,\s{G}_{1}^{\D})^* \ar[r] & H^{0}(k,\s{F}_{0}^{\D})^*,
			\end{tikzcd}
		\end{equation}
		whose rows are strict exact. By Proposition \ref{tate}, the maps $\delta^\wedge$ and $\epsilon$ are topological isomorphisms, therefore $\gamma^\wedge$ is a topological isomorphism (see \cite[Theorem 4.1]{topart}). Finally, applying Pontryagin dual to diagram \eqref{case11 i1}, we conclude that $H^1(k,\s{F}_1^D) \to (H^{1}(k,\s{F}_1)^\wedge)^*$ is also an isomorphism.
		
		For $i=2$, we have the following commutative diagram of abelian groups
		\begin{equation*}
			\begin{tikzcd}[column sep=1.5em]
				0 \ar{r} & H^{1}(k,\s{F}_{1}) \ar{r} \ar{d}{\beta} & H^{1}(k,\s{F}_2) \ar{r} \ar{d}{\gamma} & H^{1}(k,\s{G}_2) \ar{r}{\partial} \ar{d}{\delta} & H^{2}(k,\s{F}_{1}) \ar{d}{\epsilon} \\
				0 \ar{r} & H^1(k,\s{F}_{1}^{\D})^* \ar{r} & H^1(k,\s{F}_{2}^{\D})^* \ar{r} & H^1(k,\s{G}_{2}^{\D})^* \ar{r} & H^{0}(k,\s{F}_{1}^{\D})^* 
			\end{tikzcd}
		\end{equation*}
		whose rows are strict exact. The boundary map $\partial$ has finite image by Lemma \ref{boundary}, and $H^{1}(k,\s{G}_2)$ is discrete and finitely generated. Then, by Lemma \ref{profinite} and Remark \ref{exactness profinite}, we get the following commutative diagram
		\begin{equation*}
			\begin{tikzcd}[column sep=1.5em]
				0 \ar{r} & H^{1}(k,\s{F}_{1})^\wedge \ar{r} \ar{d}{\beta^\wedge} & H^{1}(k,\s{F}_2)^\wedge \ar{r} \ar{d}{\gamma^\wedge} & H^{1}(k,\s{G}_2)^\wedge \ar{r}{\partial^\wedge} \ar{d}{\delta^\wedge} & H^{2}(k,\s{F}_{1}) \ar{d}{\epsilon} \\
				0 \ar{r} & H^1(k,\s{F}_{1}^{\D})^* \ar{r} & H^1(k,\s{F}_{2}^{\D})^* \ar{r} & H^1(k,\s{G}_{2}^{\D})^* \ar{r} & H^{0}(k,\s{F}_{1}^{\D})^*,
			\end{tikzcd}
		\end{equation*}
		whose rows are strict exact. By the previous case ($i=1$) and Proposition \ref{tate}, $\beta^\wedge, \delta^\wedge$ and $\epsilon$ are topological isomorphisms, therefore, $\gamma^\wedge$ is also a topological isomorphism. As the previous point, applying Pontryagin dual, we conclude that $H^1(k,\s{F}_2^D) \to (H^{1}(k,\s{F}_2)^\wedge)^*$ is also an isomorphism. Whence we conclude the perfection and continuity of the Yoneda pairing
		\[ H^{1}(k,\s{F}_2)^\wedge \times H^1(k,\s{F}_2^{\D}) \to \Q/\Z. \]
	\end{enumerate}
	In this way, we finish the proof of the theorem.
\end{proof}

\begin{remark}
	Observe that $H^0(k,\s{F}_2^{\D})^\wedge$ does not depend on the chosen topology $\tau$. Indeed, $H^0(k,\s{F}_2^{\D})^\wedge$ is isomorphic to $(H^0(k,\s{F}_2^{\D})_{\Haus})^\wedge$ and, by Proposition \ref{top dual fil}, $H^0(k,\s{F}_2^{\D})_{\Haus}$ does not depend on $\tau$.
\end{remark}

Combining Propositions \ref{continuity yoneda}, \ref{non-degeneracy}, \ref{perfect pairing 1} and \ref{perfect pairing 2}, the main result of this thesis is obtained.

\begin{theorem} \label{Tate}
	The Yoneda pairing for the filtration pieces $\s{F}_i$
	\[ H^r(k,\s{F}_i^{\D}) \times H^{2-r}(k,\s{F}_i) \to H^2(k,\G_{m,k})\cong\Q/\Z \]
	is non-degenerate for every $i\in\{0,1,2\}$ and $r\in\Z$. Moreover, this pairing induces continuous perfect pairings between the following profinite abelian groups and discrete abelian torsion groups, respectively:
	\begin{enumerate}[nosep, label=\alph*)]
		\item $H^0(k,\s{F}_i^{\D})^\wedge$ and $H^{2}(k,\s{F}_i)$;
		\item $H^{-1}(k,\s{F}_i^{\D})^\wedge$ and $H^{3}(k,\s{F}_i)$.
	\end{enumerate}
	and induces continuous perfect pairings between the following discrete abelian torsion groups and profinite abelian groups, respectively:
	\begin{enumerate}[nosep, label=\alph*), start=3]
		\item $H^2(k,\s{F}_i^{\D})$ and $H^{0}(k,\s{F}_i)^\wedge$;
		\item $H^1(k,\s{F}_i^{\D})$ and $H^1(k,\s{F}_i)^{\wedge}$.
	\end{enumerate}
	Furthermore, these cohomology groups are trivial for $r\notin\{-1,0,1,2,3\}$.
\end{theorem}

As an immediate corollary of this theorem we get the following result, which gives a formal description of the Pontryagin dual of the algebraic Brauer group of a proper variety defined over a $p$-adic field. In particular, it gives a formal generalization of Lichtenbaum duality for any proper curve over a $p$-adic field.

\begin{corollary} \label{lichtenbaum duality}
	We have a perfect continuous pairing
	\[ H_0(X,\Z)_\tau^\wedge \times \Br_1(X) \to \Q/\Z. \]
	In particular, when $X$ curve, this pairing gives a topological isomorphism 
	\[ \Br(X)^* \cong H_0(X,\Z)_\tau^\wedge. \]
\end{corollary}

\begin{proof}
	The first assertion follows from taking $i=2$, $r=0$ in Theorem \ref{Tate} and using the fact that 
	\[ H^2(k,\s{F}_2)=H^2(k,\tau_{\leq 1}R\phi_*\G_{m,k}) \cong \Br_1(X) \]
	(see \cite[Lemma 3.1.iv]{GAunits}). The last assertion follows from the fact that the group $\Br_1(X)$ agrees with $\Br(X)$ whenever $X$ is a curve (see \cite[Proposition 5.6.1.iv]{BG}).
\end{proof}

The description of $\Br_1(X)^*$ given by the Corollary \ref{lichtenbaum duality} carries the lack of knowledge about what the group $H_0(X,\Z)_\tau$ is for an arbitrary proper $k$-variety $X$. In the next section, we give a explicit description of $H_0(X,\Z)_\tau$, relating it with $H_0(X',\Z)_\tau$, where $\nu:X'\to X$ is the normalization map.


\section{Applications}

Let $\phi:X\to\Spec k$ be a proper variety over a $p$-adic field $k$. Let $\phi':X'\to\Spec k$ be the normalization of $X$ and let $\nu:X'\to X$ be the normalization map. The pushforward $\nu_*^{(0)}:H_0(X',\Z)_\tau \to H_0(X,\Z)_\tau$ induced by $\nu$ yields the following commutative diagram of abelian groups
\begin{equation*}
	\begin{tikzcd}
		0 \ar{r} & \Ext^1_{k_{\sm}}(\Pic_{X'/k},\G_m) \ar{r} \ar{d}{\eta} & H_0(X',\Z)_\tau \ar{r}{\deg_{\phi'}} \ar{d}{\nu_*^{(0)}} & \pseudo(X')\Z \ar[d,hookrightarrow] \ar{r} & 0 \\
		0 \ar{r} & \Ext^1_{k_{\sm}}(\Pic_{X/k},\G_m) \ar{r} & H_0(X,\Z)_\tau \ar{r}[swap]{\deg_{\phi}} & \pseudo(X)\Z \ar{r} & 0,
	\end{tikzcd}
\end{equation*}
whose rows are exact (see Proposition \ref{ker coker deg}) and the right vertical arrow is injective (see Proposition \ref{divisibility pseudo index}). Then, $\ker\nu_*^{(0)}\cong\ker\eta$ and $\coker\nu_*^{(0)}$ fits in the following exact sequence

\[ 0 \to \coker \eta \to \coker\nu_*^{(0)} \to \frac{\pseudo(X)\Z}{\pseudo(X')\Z} \to 0. \]

We will give an explicit description of $\ker\eta$ and $\coker\eta$ in the case of pinched varieties.


\subsection{Pinched varieties}

Let $X'$ be a projective variety over $k$ and let $\iota':Y'\to X'$ be a closed subscheme of $X'$. Further, let $\psi:Y'\to Y$ be a finite and surjective morphism. By \cite[\S 2.1]{whichpicard}, there exists a co-cartesian and cartesian diagram of $k$-schemes
\begin{equation} \label{pinching}
	\begin{tikzcd}
		Y' \ar{r}{\iota'} \ar{d}[swap]{\psi} & X' \ar{d}{\nu} \\ Y \ar{r}[swap]{\iota} & X,
	\end{tikzcd}
\end{equation}
where $\iota$ is a closed immersion; $\nu$ is finite and surjective and induces and isomorphism $X'\backslash Y'\to X\backslash Y$; and $X$ is a proper variety over $k$. We say that $X$ is obtained by pinching $X'$ along $Y'$ via $\psi$. Furthermore, there exists an exact sequence of locally algebraic groups
\begin{equation} \label{pinched seq}
	0 \to \mu^{Y'/Y} \to \Pic_{X/k} \overset{\nu^*}{\to} \Pic_{X'/k} \to 0,
\end{equation}
where $\mu^{Y'/Y}$ is a connected linear group defined as the cokernel of $R_{Y/k}(\G_{m,Y})\overset{\psi^*}{\to}R_{Y'/k}(\G_{m,Y'})$ (see \cite[Proposition 2.4]{whichpicard}). Then, the sequence of algebraic groups
\begin{equation} \label{pinched seq 0}
	0 \to \mu^{Y'/Y} \to \Pic_{X/k}^0 \overset{\nu^*}{\to} \Pic_{X'/k}^0 \to 0
\end{equation}
is exact. Moreover, if $X'$ is geometrically normal, then $\Pic_{X'/k}^0$ is an abelian variety (see \cite[Proposition 6.2.2]{AG}) and, therefore, $L$ is the linear part of $\Pic_{X/k}^0$ (see \cite[Remark 2.6]{whichpicard}). In particular, when $k$ is $p$-adic and $X'$ is normal, the Picard variety $\Pic_{X'/k}^0$ is an abelian variety. In this case, $X'$ is the normalization of $X$. Thus, we obtain the following result.

\begin{proposition} \label{applications}
	Let $k$ be a $p$-adic field. Let $X$ be the pinching of a normal projective $k$-variety $X'$ along $Y'$ via $\psi:Y'\to Y$ (see \eqref{pinching}). Let $L$ be the group $\mu^{Y'/Y}$ in \eqref{pinched seq}. Then, $\nu:X'\to X$ induces a morphism $H_0(X',\Z)_{\tau} \to H_0(X,\Z)_{\tau}$ that fits in the following exact sequence
	\[ 0 \to \im(L^D(k)\overset{\partial}{\to}(\Pic_{X'/k}^0)^t(k)) \to H_0(X',\Z)_{\tau} \to H_0(X,\Z)_{\tau} \to F \to 0, \]
	where $\partial$ is the connecting map obtained by applying $(-)^\D$ to \eqref{pinched seq 0} and $F$ is a finite group that is an extension
	\[ 0 \to H^1(k,L^D) \to F \to \frac{\pseudo(X)\Z}{\pseudo(X')\Z} \to 0. \]
\end{proposition}

\begin{remark}
	Observe that, when $X'$ has a $k$-point, $(\Pic_{X'/k}^0)^t$ is canonically isomorphic to $\Alb_{X'/k}^0$ (see \cite[Proposition 6.2.1]{AG}). In this case, $\partial$ in Proposition \ref{applications} may be regarded as a morphism $L^D(k)\overset{\partial}{\to}\Alb_{X'/k}^0(k)$.
\end{remark}

\end{document}